\documentclass[a4paper,12pt]{amsart}

\usepackage{url}
\usepackage[colorlinks=true,hypertexnames=false,citecolor=blue,urlcolor=blue]{hyperref} 
\usepackage[backend=biber, style=alphabetic]{biblatex} 
\IfFileExists{csquotes.sty}{\usepackage{csquotes}}{}

\usepackage{amsfonts, amsmath, amssymb} 
\providecommand{\mathscr}{\mathcal}
\usepackage{graphicx, xcolor} 
\usepackage{enumerate} 
\usepackage[english]{babel}
\IfFileExists{blindtext.sty}{\usepackage{blindtext}}{}
\usepackage[margin=1in, top=1.1in, bottom=1.1in]{geometry}

\newtheorem{thm}{Theorem}[section]

\newtheorem{lem}{Lemma}[thm]
\newtheorem{prop}{Proposition}[thm]

\numberwithin{equation}{section}

\theoremstyle{definition}
\newtheorem{definition}{Definition}[thm]
\newtheorem{rem}{Remark}[thm]
\newtheorem{exa}{Example}[thm]

\DeclareMathOperator{\Id}{Id}

\DeclareMathOperator{\FT}{\mathcal{F}}
\DeclareMathOperator{\FTh}{\widehat{\mathcal{F}}}
\DeclareMathOperator{\M}{\mathcal{M}}
\DeclareMathOperator{\Mh}{\widehat{\mathcal{M}}}
\DeclareMathOperator{\sinc}{sinc}
\DeclareMathOperator{\D}{\mathcal{D}}
\DeclareMathOperator{\LL}{\mathcal{L}}
\DeclareMathOperator{\ZZ}{\mathcal{Z}}
\DeclareMathOperator{\Dh}{\widehat{\mathcal{D}}}
\DeclareMathOperator{\HH}{\mathcal{H}}
\DeclareMathOperator{\VN}{VN}

\DeclareMathOperator{\G}{\mathbb{G}}
\DeclareMathOperator{\Gh}{\widehat{\mathbb{G}}}
\DeclareMathOperator{\coproduct}{\Delta}

\begin{document}

\title{H\"ormander-Mikhlin type theorem on non-commutative spaces}

\author[R. Akylzhanov]{Rauan Akylzhanov}

\address{
Rauan Akylzhanov:
 \endgraf
 Former address: School of Mathematical Sciences, Queen Mary University of London, London,
 \endgraf
 UK
 \endgraf
  {\it E-mail address} {\rm akylzhanov.r@gmail.com}
  }

\author[M. Ruzhansky]{Michael Ruzhansky}

\address{
 Michael Ruzhansky:
  \endgraf
 Department of Mathematics: Analysis, Logic and Discrete Mathematics,
  \endgraf
 Ghent University, Ghent,
 \endgraf
  Belgium
  \endgraf
  and
 \endgraf
 School of Mathematical Sciences, Queen Mary University of London, London,
 \endgraf
 UK
 \endgraf
  {\it E-mail address} {\rm michael.ruzhansky@ugent.be}
  }

\author[K. Tulenov]{Kanat Tulenov}

\address{
Kanat Tulenov:
\endgraf
Department of Mathematics: Analysis, Logic and Discrete Mathematics
\endgraf
Ghent University, Ghent,
\endgraf
Belgium
\endgraf
Institute of Mathematics and Mathematical Modeling, 050010, Almaty,
\endgraf
Kazakhstan
\endgraf
{\it E-mail address} {\rm kanat.tulenov@ugent.be}
}
\begin{abstract}
In this paper, we introduce a Fourier-type formalism on non-commutative spaces. As a result, we obtain two versions of H\"ormander-Mikhlin $L^{p}$-multiplier theorem: on general von Neumann algebras and semi-finite von Neumann algebras respectively. We develop elements of noncommutative Littlewood-Paley theory to connect with classical symbolic estimates. In the particular case of $\mathbb{R}$, we explicitly show how to recover a sharp version of the classical H\"ormander $L^{p}$-multiplier theorem, which was obtained by Grafakos and Slav\'ikov\'a in \cite{GS19}. Finally, we apply the main results to obtain time-asymptotics of abstract Klein-Gordon equation on semi-finite von Neumann algebras.
\end{abstract}
\subjclass[2020]{46L51, 47L25, 42B15, 43A85, 43A15, 42A38, 43A32, 35S05.}

\keywords{Non-commutative space, semifinite von Neumann algebra, Fourier multiplier, H\"ormander-Mikhlin theorem, spectral multipliers.}

\maketitle

\section{Introduction}

Let $\sigma$ be a bounded function on $\mathbb{R}^{n}.$ Then for any function $f$ in the Schwartz space $\mathcal{S}(\mathbb{R}^{n})$ define a linear operator $A_{\sigma}$ on $\mathcal{S}(\mathbb{R}^{n})$ as follows
\begin{equation}\label{Fourier-multip}
A_{\sigma}f(\xi)=\int_{\mathbb{R}^n}\widehat{f}(\eta)\sigma(\eta)e^{2\pi i(\xi,\eta)}d\eta,
\end{equation}
where
$\widehat{f}$ is the Fourier transform
$$\widehat{f}(\xi)=\int_{\mathbb{R}^n}f(\eta)e^{-2\pi i(\xi,\eta)}d\eta, \quad \xi\in \mathbb{R}^n,$$ of the function $f,$ and $(\xi,\eta)$ denotes the standard inner product of vectors $\xi$ and $\eta$ in $\mathbb{R}^n.$
One of the main problems in harmonic analysis is to find a sufficient condition on $\sigma$ that guarantees that the operator $A_{\sigma}$ is bounded on $L^p(\mathbb{R}^n)$ for a given $1<p<\infty.$ In this case, the operator $A_{\sigma}$ will be called Fourier multiplier with the symbol $\sigma.$ In \cite{M56}, Mikhlin showed that if the condition
$$|\partial^{\alpha}\sigma(\xi)|\leq C_{\alpha}|\xi|^{-|\alpha|}, \quad 0\neq \xi\in\mathbb{R}^n,$$
holds for any multi-indices $\alpha$ such that $|\alpha|\leq [\frac{n}{2}]+1,$ then the operator $A_{\sigma}$ has a bounded extension from $L^p(\mathbb{R}^n)$ into itself for any $1<p<\infty.$ Later, H\"ormander  extended in \cite{Hor} Mikhlin's result. Indeed, H\"ormander's theorem says that if $(I-\Delta_{\mathbb{R}^{n}})^{s/2}$  is the operator given on the Fourier transform by multiplication by $(1+4\pi^{2}|\xi|^{2})^{s/2}$ for some $s>0$ and let $\Psi$ be a Schwartz function on $\mathbb{R}^{n}$ whose Fourier transform is supported in the annulus $\frac{1}{2}<|\xi|<2$ and satisfies $\sum\limits_{j \in \mathbb{Z}} \widehat{\Psi}\left(2^{-j} \xi\right)=1,\xi\neq 0$. Moreover, if for some $1\leq r \leq 2$ and $s>n/r,$ the function $\sigma$ satisfies the condition
\begin{equation}\label{eq:horm_cond}
\sup_{j \in \mathbb{Z}} \left\|(I-\Delta_{\mathbb{R}^{n}})^{\frac{s}{2}}\left[\widehat{\Psi} \sigma(2^{j} \cdot)\right]\right\|_{L^{\frac{n}{s},\,1}\left(\mathbb{R}^{n}\right)}<\infty.
\end{equation}
This is a Littlewood--Paley formulation: the symbol is localised to dyadic annuli via $\sigma(2^j\cdot)$ and a smooth $\Psi$ supported in $\tfrac{1}{2}<|\xi|<2$ with $\sum_j \widehat{\Psi}(2^{-j}\xi)=1$; it is the standard way to verify multiplier bounds in the Euclidean setting.
Then $A_{\sigma}$ has a bounded extension from $L^p(\mathbb{R}^n)$ into itself for any $1<p<\infty.$ In \cite{CT77}, Cald\'eron and
Torchinsky obtained that if condition \eqref{eq:horm_cond} holds with $p$ satisfying
$\left|\frac{1}{p}-\frac{1}{2}\right|<\frac{s}{n}$ and $\left|\frac{1}{p}-\frac{1}{2}\right|=\frac{1}{r},$ then $A_{\sigma}$ is bounded from $L^p(\mathbb{R}^n)$ into itself. Very recently, Grafakos and Slav\'ikov\'a in \cite{GS19}  improved H\"ormander's theorem and the result of Cald\'eron and Torchinsky by replacing the norm in \eqref{eq:horm_cond} with the Lorentz space $L^{\frac{n}{s}, 1}\left(\mathbb{R}^{n}\right)$-quasi-norm. Their main result is given as follows:

\begin{thm}\label{thm_grafakos} \cite[Theorem 1.1]{GS19}. Let $\Psi$ be a Schwartz function on $\mathbb{R}^{n}$ whose Fourier transform is supported in the annulus $\frac{1}{2}<|\xi|<2$ and satisfies $\sum\limits_{j \in \mathbb{Z}} \widehat{\Psi}\left(2^{-j} \xi\right)=1, \quad \xi \neq 0$. Let $1<p<\infty$ and let $s \in(0, n), n \in \mathbb{N}$, satisfy
$$
\left|\frac{1}{p}-\frac{1}{2}\right|<\frac{s}{n}.
$$
Then for all functions $f \in \mathcal{S}\left(\mathbb{R}^{n}\right)$ we have the following a priori estimate
\begin{equation}\label{thm_grafakos_eq}
\|A_{\sigma}f\|_{L^{p}(\mathbb{R}^{n})} \leq C 
\sup_{j \in \mathbb{Z}}
\left\|(I-\Delta_{\mathbb{R}^{n}})^{\frac{s}{2}}\left[\widehat{\Psi} \sigma(2^{j} \cdot)\right]\right\|_{L^{\frac{n}{s}, 1}\left(\mathbb{R}^{n}\right)}
\|f \|_{L^{p}(\mathbb{R}^{n})}
\end{equation}
with some constant $C>0.$
\end{thm}
The main aim of this paper is to obtain variants of the H\"ormander $L^{p}$-multiplier theorem on general von Neumann algebras.

Let $\M$ be a von Neumann algebra and let a linear operator $\D\colon L^{2}(\M) \rightarrow L^{2}(\M)$ be affiliated with the dual von Neumann algebra $\Mh$.

We prove the $L^p(\M)\to L^p(\M)$ boundedness of Fourier multipliers in two complementary forms: a \emph{global} estimate (operator norm on $\widehat{\M}$) and a \emph{local} estimate (Littlewood--Paley decomposition, fibre-wise norms), the latter recovering the classical annulus condition of Grafakos--Slav\'ikov\'a \cite{GS19} on $\mathbb{R}^n$ (Example~\ref{EX_case_of_Rn}).
The \emph{global form} (Section~\ref{sec:mihlin}, Theorem~\ref{THM:Mihlin-Hormander-general}) takes the form:
\begin{equation}\label{1.2}
\|A\|_{L^{p}(\M) \rightarrow L^{p}(\M)}
\leq
\left\|\Dh^{-\beta}\right\|_{L^{\infty}(\Mh)}
\left\|\Dh^{\beta} A\right\|_{L^{r}(\Mh)}, \quad \frac{1}{r}=2\left\lvert\frac{1}{p}-\frac{1}{2}\right\rvert, \quad 1<p<\infty.
\end{equation}
The \emph{local (Littlewood--Paley) form} (Section~\ref{SEC_noncommutative_littlewood_paley}, Theorem~\ref{THM:Mihlin-Hormander-general-noncomm-L-Paley}) takes the following form
under a direct-integral decomposition $\widehat{\M} = \int_{\widehat{Z}}^\oplus \Mh_\lambda\,d\mu(\lambda)$ and a Littlewood--Paley decomposition over $\widehat{Z}$; here the $T_j\colon \widehat{Z}\to \widehat{Z}$ are the measurable maps of that decomposition (so $T_j^{-1}(E_j)=E_0$ for all $j$ and they have uniformly bounded Radon--Nikodým derivatives; see Section~\ref{SEC_noncommutative_littlewood_paley}).
\begin{equation}\label{intro:LP-local}
\|A\|_{L^{p}(\M) \rightarrow L^{p}(\M)}
\lesssim
\sup_{j\in\mathbb{Z}}
\Bigl\|\bigl\|\LL(\lambda)A(T_j(\lambda))\bigr\|_{L^{r,\infty}(\Mh_\lambda)}\Bigr\|_{L^{r,1}(\widehat{Z})}, \quad \frac{1}{r}=2\left\lvert\frac{1}{p}-\frac{1}{2}\right\rvert, \quad 1<p<\infty.
\end{equation}
Here, we choose $\beta>0$ such that $\left\|\Dh^{-\beta}\right\|_{L^{\infty}(\Mh)}$ and the relevant $L^r(\Mh)$-norms are finite; the local form requires the LP decomposition setup of Section~\ref{SEC_noncommutative_littlewood_paley}.
As a particular example, we recover Theorem \ref{thm_grafakos} of \cite{GS19}  in Example \ref{EX_case_of_Rn}.

For an $\LL$-wave equation $\frac{\partial^2{u}}{{\partial t^2}} + \LL u = 0, \partial u=u_1$, we prove a time-decay estimate for the propagator 
$u(t) = \frac{\sin(t\cdot \sqrt{\LL})}{\sqrt{\LL}} u_1 + \cos(t\sqrt{\LL}) u_0$.
$\LL^{-\beta}\frac{\sin t\sqrt{\LL}}{\sqrt{\LL}}$ (Theorem~\ref{THM_time_rescaling}): under a Weyl law $\tau(E_{(\kappa,s)}(\LL))\cong s^\alpha$ and $\frac{1}{r}\geq \frac{\beta}{\alpha}$, $0<\gamma\leq 2$,
\begin{equation}\label{intro:wave-decay}
\|u(t)\|_{L^q(\M)}
\lesssim
(C_{\gamma,p,q,\beta,\kappa,\LL}\cdot t^{1-\gamma\beta -\frac{\alpha}{\beta}\gamma \frac{1}{r}} + 1) (\|\LL^{\beta} u_1\|_{L^p(\M)} + \|u_0\|_{L^p(\M)}).
\end{equation}

The structure of the paper is as follows.
Section~\ref{sec:preliminaries} introduces preliminaries on noncommutative $L^p$ and Lorentz spaces and the notion of Fourier structure on von Neumann algebras; it is used in all subsequent sections.
Section~\ref{sec:paley} develops Hausdorff--Young and Paley-type inequalities under the Fourier structure; these feed into Section~\ref{sec:hardy-littlewood} (Hardy--Littlewood) and into Section~\ref{sec:mihlin} (multiplier theorems).
Section~\ref{sec:hardy-littlewood} establishes Hardy--Littlewood-type inequalities involving $\D^{-\beta}$ and the dual theorem (Theorem~\ref{THM:dual-HLP}); the latter is the key step in the H\"ormander--Mihlin type proofs in Section~\ref{sec:mihlin}.
Section~\ref{sec:mihlin} proves the main H\"ormander--Mihlin multiplier theorems (off-diagonal $L^p$--$L^q$ and diagonal semifinite Lorentz version); it relies on Sections~\ref{sec:preliminaries}--\ref{sec:hardy-littlewood} and gives operator-norm bounds.
Section~\ref{SEC_noncommutative_littlewood_paley} develops a Littlewood--Paley framework (reduction theory, direct integrals, decompositions over the spectrum of an abelian subalgebra) so that multiplier boundedness can be stated in symbolic/fibre form (Theorem~\ref{THM:Mihlin-Hormander-general-noncomm-L-Paley}); it uses the same Fourier structure and multiplier theorems, and Example~\ref{EX_case_of_Rn} shows that this formulation recovers the classical Grafakos--Slav\'ikov\'a condition \cite{GS19} on $\mathbb{R}^n$.
Section~\ref{sec:applications} presents applications to evolution equations (heat and wave propagators), using the multiplier theorems and Lorentz norms; it recovers results from \cite{GS19, FR21, KR2023, DVRT2025}.

We also note that the $L^{p}$-boundedness of Fourier multipliers can be extended to the range $0<p<\infty$ in the context of Hardy spaces; see the results on general graded groups in \cite{HHR23} and references therein.

\section{Preliminaries}
\label{sec:preliminaries}

\subsection{\texorpdfstring{Classical $L^p$ and $L^{p,q}$ spaces}{Classical Lp and Lpq spaces}}
For $0< p\leq\infty$ and $\mathbb{R}_{+}:=(0,\infty),$ the space $L^p(\mathbb{R}_{+})$ refers to the $L^p$-space of equivalence classes of $p$-integrable functions, defined almost everywhere, and $L^{\infty}(\mathbb{R}_{+})$ denotes the space of essentially bounded functions on the positive half-line $\mathbb{R}_{+}.$ For $0<  p,q\leq\infty,$ the Lorentz space $L^{p,q}(\mathbb{R}_{+})$ consists of all complex-valued measurable functions $f$ on $\mathbb{R}_{+}$ for which the following quasi-norm is finite:

$$\|f\|_{L^{p,q}(\mathbb{R}_{+})}=\left\{ \begin{array}{rcl}
         \left(\int\limits_{\mathbb{R}_{+}}\left(t^{\frac{1}{p}}f^*(t)\right)^q\frac{dt}{t}\right)^{\frac{1}{q}}, & \mbox{for}
         & q<\infty; \\ \sup\limits_{t>0}t^\frac{1}{p}f^*(t),\;\;\;\;\;\;\;\;\;\;\;\; & \mbox{for} & q=\infty.
                \end{array}\right.
$$
Here, $f^*$ denotes the decreasing rearrangement of the function $f.$ The quasi-norm in $L^{p,\infty}(\mathbb{R}_{+})$ is also given as follows
\begin{equation}\label{weak-quasi-norm}
    \|f\|^p_{L^{p,\infty}(\mathbb{R}_{+})}=\sup\limits_{t>0}t^p m\{s>0:|f(s)|>t\},
\end{equation}
where $m$ is the Lebesgue measure on $\mathbb{R}_{+}.$ For more information on these spaces and their properties, we refer the reader to \cite{G2008}.
\subsection{\texorpdfstring{Von Neumann algebras and $\tau$-measurable operators}{Von Neumann algebras and tau-measurable operators}}
Let $\HH$ be a complex Hilbert space, and let $B(\HH)$ denote the algebra of all bounded linear operators on $\HH.$ Consider a semi-finite von Neumann algebra $\M\subset B(\HH)$ equipped with a faithful, normal, semi-finite trace $\tau.$ The pair $(\M,\tau)$ is referred to as a \textit{non-commutative measure space}. A closed, densely defined operator $x$ on $\HH$ with domain $\mathrm{dom}(x),$ is called \textit{affiliated} with $\M$ if $u^{\ast}xu=x$ for every unitary operator $u$ in the commutant $\M'$ of $\M$. Such an operator
$x$ is called \textit{$\tau$-measurable} if $x$ is affiliated with $\M$ and, for every $\varepsilon>0$ there exists a projection $p\in \M$ such that $p(\HH)\subset \mathrm{dom}(x)$ and $\tau(\mathbf{1}-p)<\varepsilon.$
The set of all $\tau$-measurable operators is denoted by $L^0(\M)$.

For any element \( x \in L^0(\M) \), let \( x = U |x| \) be its polar decomposition, where \( |x| = (x^* x)^{1/2} \) is the unique positive part. The spectral projection of \( |x| \) corresponding to \( (s, +\infty) \) is denoted by \( E_{(s, +\infty)}(|x|) \).
 The \textit{distribution function} of $x$ is defined as
\begin{equation}
\label{EQ_distribution_function}
d(s;|x|)=\tau(E_{(s, +\infty)}(|x|)), \quad 0\leq s<\infty.
\end{equation}
The function $d(\cdot;|x|):[0,\infty)\rightarrow[0,\infty]$ is non-increasing, and the normality of the trace ensures that
$d(s;|x|)$ is right-continuous.
The \textit{generalized singular value function} of $x$ is given by
$$
\label{EQ_mu_t_distribution_function}
\mu(t;x)=\inf\left\{s>0: d(s;|x|)\leq t\right\}, \quad t>0.
$$
The function $t\mapsto\mu(t;x)$ is decreasing, right-continuous, and serves as the right inverse of the distribution function of $x.$
Moreover, for $x,y \in L^0(\M,\tau),$ the generalized singular value function satisfies (see, for example, \cite[Lemma 2.5 (vii)]{FK}, \cite[Corollary 2.3.16 (b)]{LSZ}):
\begin{equation}\label{sub-multiplicity}
\mu(t+s;xy)\leq\mu(t;x)\cdot\mu(s;y),\quad t,s>0.
\end{equation}

For further details on generalized singular value functions and their properties, we refer the reader to
\cite{DPS, FK, LSZ}.

\subsection{\texorpdfstring{Noncommutative $L^p$ and $L^{p,q}$ spaces associated with a semi-finite von Neumann algebra}{Noncommutative Lp and Lpq spaces associated with a semi-finite von Neumann algebra}}

For $0<p<\infty$, define the noncommutative $L^p$-space
$$L^{p}(\M)=\{x\in L^{0}(\M):\tau\left(|x|^{p}\right)<\infty\},$$
where,
$$\|x\|_{L^{p}(\M)}=\left(\tau\left(|x|^{p}\right)\right)^{\frac1p}, \,\;x\in L^{p}(\M).$$
Then, the space $\left(L^{p}(\M), \|\cdot\|_{L^{p}(\M)}\right)$ is a quasi-Banach (Banach for $p\geq 1$) space. They are the
noncommutative $L^{p}$-spaces associated with $(\M,\tau)$, denoted by $L^{p}(\M,\tau)$ or simply
by $L^{p}(\M)$. As usual, we set $L^{\infty}(\M,\tau)=\M$ equipped with the operator norm. Moreover, we have $$\|x\|_{L^{\infty}(\M)}=\|\mu(x)\|_{L^{\infty}(0,\infty)}=\mu(0;x),$$ $x\in L^{\infty}(\M).$

\begin{definition}Let $0 < p, q\leq \infty$ and let $\M$ be a semi-finite von Neumann algebra. Then, define the non-commutative Lorentz $L^{p,q}(\M)$ space by
$$L^{p,q}(\M):=\{x\in L^{0}(\M):\|x\|_{L^{p,q}(\M)}<\infty\},$$
where
$$
\|x\|_{L^{p,q}(\M)}:=\left(\int_{
\mathbb{R}_+}\left(t^\frac{1}{p}\mu(t;x)\right)^{q}\frac{dt}{t}\right)^{\frac1q},\,\ \text{for} \,\ q<\infty,
$$
and
$$\|x\|_{L^{p,q}(\M)}:=\sup_{t>0}t^{\frac{1}{p}}\mu(t;x), \,\ \text{for} \,\ q=\infty.$$
In other words, we have
\begin{equation}\label{NC Lorentz space norm II}\|x\|_{L^{p,q}(\M)}=\|\mu(x)\|_{L^{p,q}(\mathbb{R}_+)}.
\end{equation}
In particular,
$L^{p,p}(\M)=L^p(\M)$ isometrically for $1<p<\infty.$
\end{definition}
For more details on noncommutative spaces associated with von Neumann algebras, we refer the reader to \cite{DPS, LSZ, PXu}.
The following is a well-known result for the classical Lorentz $L^{p,q}$ spaces which plays a key role in our investigation.
\begin{prop}\label{Holder-ineq-Lorentz-spaces} \cite[Exercises 1.4.19, p. 73]{G2008} \cite[Theorem 2.4.] {kosaki1981non} Let $\M$ be a semi-finite von Neumann algebra and $0<p,q,r\leq \infty,$ $0<s_{1},s_{2}\leq\infty,$    $\frac{1}{p}+\frac{1}{q}=\frac{1}{r}$ and $\frac{1}{s_1}+\frac{1}{s_2}=\frac{1}{s}.$ Then there exists a constant $C_{p,q,s_1,s_2}>0$ (which depends on $p, q,$ $s_1,$ and  $s_{2}$) such that,
\begin{eqnarray}\label{Lorentz-holder}
\|fg\|_{L^{r,s}(\M)}\leq C_{p,q,s_1,s_2}\|f\|_{L^{p,s_1}(\M)}\|g\|_{L^{q,s_2}(\M)}, f\in L^{p,s_1}(\M), g\in L^{q,s_2}(\M).
\end{eqnarray}
\end{prop}
%
%
%
Following \cite[Proposition 1.4.9]{G2008} and properties of semi-finite traces on spectral projections, we immediately get
\begin{prop}
\label{PROP_distribution_function}
Let $\M$ be a semi-finite von Neumann algebra and let $0< p < \infty$ and $0<q\leq \infty$. Then we have the identity
\begin{equation}
\label{EQ_distribution_function_Lorentz}
\|x\|_{L^{p,q}(\M)} =
\begin{cases}
p^{1/q} \left( \int\limits_{0}^{\infty} \big[ d(s;|x|)^{1/p} s \big]^q \, \frac{ds}{s} \right)^{1/q}, & \text{when } q < \infty, \\[1.5ex]
\sup\limits_{\substack{s>0}} \, s \, d(s;|x|)^{1/p}, & \text{when } q = \infty .
\end{cases}
\tag{1.4.5}
\end{equation}
\end{prop}

\subsection{\texorpdfstring{Noncommutative $L^p$ spaces associated with general von Neumann algebras}{Noncommutative Lp spaces associated with general von Neumann algebras}}
In general, a von Neumann algebra need not admit a trace.
In particular, type III von Neumann algebras — which are characterized by the absence of nontrivial finite projections — admit no nontrivial normal tracial states \cite{takesaki1972operator}.

It can be seen that type III von Neumann algebras admit no semi-finite faithful normal trace. For the history of $L^p$-spaces on general von Neumann algebras we refer to \cite{PXu}.

For a general von Neumann algebra that might only admit a weight, Haagerup used the crossed-product of $M$ with respect to its modular automorphism group \cite{haagerup1979lp}.
However, it can be shown that the Haagerup $L^p$-spaces are not compatible with respect to the complex interpolation functor \cite{terp1981lp}, i.e. their natural intersection is empty for two different indices $p_0\neq p_1$.
To overcome this limitation, Kosaki introduced \cite{kosaki1984applications} a new way to construct interpolation pairs by using Radon-Nikodym derivative of trace on the crossed-product with respect to the weight.
He also showed that his space is isometrically isomorphic to Haagerup's.  In \cite{terp1981lp}, Terp extended \cite{terp1982interpolation} a special case of Kosaki's complex interpolation construction to the semi-finite case.
We shall follow the construction by Izumi \cite{izumi1997constructions} which generalizes both \cite{kosaki1984applications}, \cite{terp1981lp,terp1982interpolation}, \cite{hilsum1981espaces} and \cite{haagerup1979lp}.


By considering two linear operators $A_0B \mapsto A_0B$ and $AB_0\mapsto AB_0$, where we vary $A$ for fixed $B_0$ and vary $B$ for fixed $A_0$ and applying the complex interpolation functor, we obtain the following
\begin{prop}
\label{PROP:izumi_holder_inequality}
Let $\M$ be a general von Neumann algebra and let $0<p,q,r\leq \infty$ with $\frac{1}{p}+\frac{1}{q}=\frac{1}{r}$. Then  we have
\begin{equation}
\label{EQ:holder-inequality}
\|f \cdot g\|_{L^r(\M)} \lesssim \|f\|_{L^p(\M)} \cdot \|g\|_{L^q(\M)}.
\end{equation}
\end{prop}
We make the following remark to agree on the affiliation relation on general (possibly non-tracial) von Neumann algebras. 
\begin{rem}
Let $\M$ be a general von Neumann algebra.  
In the semi-finite case, we consider closed, densely defined linear operators affiliated with $\M$ in the usual sense.
However, if $\M$ is not semi-finite, then we instead consider its crossed product $\widetilde{\M} := \M \rtimes_{\sigma^\varphi} \mathbb{R}$ with respect to the modular automorphism group $\sigma_{\omega}$ of a faithful normal weight $\omega$, and affiliation is taken with respect to $\widetilde{\M}$. This convention ensures a unified treatment of measurable operator theory even in the absence of a trace.
In both cases, we shall continue to write $A \,\eta\, \M$ to denote affiliation, by a slight abuse of notation. The set of all $\tau$-measurable operators will still be denoted by $L^0(\M)$.
\end{rem}

\subsection{Fourier structure on von Neumann algebras}
Let $\M$ and $\Mh$ be two von Neumann algebras with normal semi-finite faithful weights $\omega$ and $\widehat{\omega},$ respectively. We denote by $\M_{*}$ the predual of $\M$, i.e. $\left(\M_{*}\right)^{*} \cong \M$.
For each weight $\omega,$ we can associate a Hilbert space $L^{2}(\M)$ through the Gelfand-Naimark-Segal (GNS) construction. 

Since we will be formulating the Fourier transform in terms of the semi-finite traces, let us recall the definition of the dense subspace $\mathcal{I}$ originating from the details of the GNS construction, i.e.
\begin{equation}
\mathcal{I} = \{\omega\in \M_* \colon \text{there exists } M >0 \text{ such that } |\omega(x^*)| \leq M\|\Lambda(x)\|,\quad x\in \mathcal{N}_{\omega}\}.
\end{equation}

\subsection{The Gelfand-Naimark-Segal construction}
Let $\omega$ be a normal, faithful (semi-finite) weight on a von Neumann algebra $\M$. The GNS construction provides a triple 
\[
(L^2(\M,\omega), \pi, \Lambda)
\]
where:

\begin{enumerate}
    \item \textbf{Hilbert Space:} $L^2(\M,\omega)$ is the completion of the set 
    \[
    \mathcal{N}_\omega = \{x\in \M : \omega(x^*x)<\infty\}
    \]
    with respect to the inner product defined by
    \[
    \langle \Lambda(x), \Lambda(y)\rangle = \omega(y^*x) \quad \text{for all } x,y\in\mathcal{N}_\omega.
    \]
    
    \item \textbf{Representation:} The mapping 
    \[
    \pi : \M \to B(L^2(\M,\omega))
    \]
    is defined by left multiplication:
    \[
    \pi(x)\Lambda(y)=\Lambda(xy) \quad \text{for all } x\in \M \text{ and } y\in \mathcal{N}_\omega.
    \]
    Since $\omega$ is faithful, $\pi$ is a faithful *-representation.
    
    \item \textbf{Canonical Map:} $\Lambda : \mathcal{N}_\omega \to L^2(\M,\omega)$ is the natural map associating each $x\in \mathcal{N}_\omega$ with its equivalence class in the Hilbert space.
\end{enumerate}

Thus, by identifying $\M$ with its image $\pi(\M)$, we realize $\M$ as a von Neumann subalgebra of $B(L^2(\M,\omega))$, the algebra of bounded operators on the Hilbert space $L^2(\M,\omega)$.

\begin{definition}
\label{Fourier-structure}(Fourier structure).
Let $\M$ and $\Mh$ be a pair of general von Neumann algebras. Then we shall say that the tuple $(\M, \Mh)$ admits a Fourier structure if:
\begin{itemize}
\item  There is a linear map $\FT_{\Mh}: \mathcal{N}_{\tau}(\M) \to \mathcal{N}_{\widehat{\tau}}(\Mh)$ such that
\begin{equation}
\label{FT-l1-l-infty}
\left\|\FT_{\Mh}[x]\right\|_{\Mh} \leq\|x\|_{\M_{*}},\quad x \in \mathcal{N}_{\omega}(\M)
\end{equation}

\item  There is a linear map $\FTh_{\M}: \mathcal{N}_{\tau}(\Mh) \to \mathcal{N}_{\widehat{\omega}}(\M)$ such that
\begin{equation}
\left\|\FTh_{\M}(x)\right\|_{\M} \leq\|x\|_{\Mh_*},\quad x \in \mathcal{N}_{\widehat{\tau}}(\Mh)
\end{equation}
\item The linear maps $\FT_{\Mh}$ and $\FTh_{\M}$ are inverse to each other in the following way:
\begin{eqnarray}
    \FTh_{\M}(\FT_{\Mh}(x)) = x, \quad x\in \mathcal{N}_{\omega}(\M), \quad  \label{eq:fourier1} \\
    \FT_{\Mh}(\FTh_{\M}(x)) = x, \quad x\in \mathcal{N}_{\widehat{\omega}}(\Mh). \quad \label{eq:fourier2}
\end{eqnarray}

\item The following equalities hold
\begin{eqnarray}\label{Plancherel}
\widehat{\omega}(\FT_{\Mh}(x)^*\FT_{\Mh}(x)) =  \omega(x^*x) \quad x \in \mathcal{N}_{\omega}(\M)  \\
\omega(\FTh_{\M}(x)^*\FTh_{\M}(x)) =  \widehat{\omega}(x^*x) \quad x \in \mathcal{N}_{\widehat{\omega}}(\Mh)
\end{eqnarray}

\item For any linear operator  $A$ affiliated with $\Mh$, we have the following
\begin{equation}
\label{eq:affiliation}
\FT_{\Mh}(A x) = A \cdot \FT_{\Mh}[x],\quad x \in \mathcal{N}_{\omega}(\M)
\end{equation}
\item For any linear operator  $A$ affiliated with $\M$, we have the following
\begin{equation}
\label{eq:affiliation-dual}
\FTh_{\M}(A x) = A \cdot \FTh_{\M}[x],\quad x \in \mathcal{N}_{\widehat{\omega}}(\Mh)
\end{equation}

\end{itemize}
\end{definition}

\begin{rem} The concept of defining a Fourier multiplier on general semi-finite von Neumann algebras $\M$ arises from the paper \cite{AR}. Specifically, for a locally compact separable unimodular group $G$ and its associated group von Neumann algebra $VN(G),$ an operator
$A$ is considered a left Fourier multiplier on $G$ if and only if $A$ is affiliated with the right von Neumann algebra
$VN_R(G)$ and is $\tau$-measurable. Operators affiliated with $VN_R(G)$ are exactly those that remain left-invariant under the action of $G,$ meaning that
$$A\pi_L(g)=\pi_L(g) A, \quad \text{for all}\quad g\in G,$$
where $\pi_L(g)$ represents the left action of $G$ on $L^2(G).$ In other words, left Fourier multipliers on
$G$ are exactly the left-invariant operators that are also
$\tau$-measurable. For more details, one can refer to \cite[Definition 2.16]{AR} and \cite[Remark 2.17]{AR}.
\end{rem}
Below, we provide two examples that satisfy the Fourier structure.
First, we show that on spaces without symmetries (see Example \ref{EX:L-decomposition}), the presence of a suitable linear operator $\D$ is sufficient to study metric properties of Fourier operators. In the second example, the presence of extra coproduct operator on involutive algebra makes it possible to fulfill conditions of Definition~\ref{Fourier-structure}.
\begin{exa}[Fourier structure generated by an affiliated operator \(\D\)]
\label{EX:L-decomposition}
Let \(\M\subset B(\HH)\) be a semi-finite von Neumann algebra on \(\HH\) with a normal faithful semifinite trace \(\tau\) and GNS space \(L^2(\M)\) (identified with \(\HH\)). Let $\D$ be affiliated with $\M$ and \(\langle|\D|\rangle\) the abelian algebra generated by the spectral projections of \(|\D|\) from \(\D=U|\D|\).
\noindent\textbf{(1) Direct integral and dual algebra.}
By the spectral theorem there is a unitary \(W:\HH\to\int_{\mathrm{sp}(|\D|)}^\oplus\HH^\lambda\,d\nu(\lambda)\) such that \(W|\D|W^{-1}=\int^\oplus\lambda\,\Id_{\HH^\lambda}\,d\nu(\lambda)\). Then \(\langle|\D|\rangle\) acts by multiplication in \(\lambda\), and the commutant \(\widehat{\M}:=[\langle\D\rangle]'\cong\int_{\mathrm{sp}(|\D|)}^\oplus B(\HH^\lambda)\,d\nu(\lambda)\) acts fibrewise. The maps \(\FT_{\Mh},\FTh_{\M}\) are implemented by \(W\), so (iii)--(iv) and \eqref{FT-l1-l-infty} in Definition~\ref{Fourier-structure} hold.
\noindent\textbf{(2) $L^2$- and $L^1$-Fourier transform.}
Set \(\FT_{\Mh}[f]:=Wf\) and \(\widehat f(\lambda):=(Wf)(\lambda)\). Then \(\|f\|_{L^2(\M)}^2=\int_{\mathrm{sp}(|\D|)} \|\widehat f(\lambda)\|_{\HH^\lambda}^2\,d\nu(\lambda)\) and \(f=W^*\widehat f\). For the $L^1$--$L^\infty$ map, the expansion in generalized eigenfunctions \cite{LT16} yields a map \(\M_*\to\widehat{\M}\) with \(\|\widehat{\phi}(\lambda)\|_{B(\HH^\lambda)}\le\|\phi\|_{\M_*}\) for \(\phi\in\M_*\).
\noindent\textbf{(3) Multiplier property.}
For any Borel \(m:\mathbb{R}\to\mathbb{C}\), \(W m(\D) W^*=\int^\oplus m(\lambda)\,\Id_{\HH^\lambda}\,d\nu(\lambda)\), so \(\widehat{m(\D)f}(\lambda)=m(\lambda)\widehat f(\lambda)\); thus \(m(\D)\) is a Fourier multiplier with symbol \(m(\lambda)\) (item (vi) of Definition~\ref{Fourier-structure}).
\noindent\textbf{(4) Discrete vs.\ continuous spectrum.}
If \(\mathrm{sp}(\D)=\{\lambda_k\}_{k\in\mathbb{N}}\), then \(\HH\cong\bigoplus_k \HH^{\lambda_k}\) and \(\widehat{\M}\cong\prod_k B(\HH^{\lambda_k})\) (block-diagonal as on compact manifolds). If \(\mathrm{sp}(\D)\subset\mathbb{R}\) with \(\nu\ll d\lambda\), \(\widehat{\M}\) is a direct integral over \(\lambda\). For \emph{normal} \(\D\), the same holds with \(\mathrm{sp}(\D)\subset\mathbb{C}\); the model is unique up to unitary equivalence of the fibre realisation.
Thus \((\M,\widehat{\M})\) and \(\FT_{\Mh},\FTh_{\M}\) satisfy Definition~\ref{Fourier-structure}, and every \(m(\D)\) is a Fourier multiplier with symbol \(m(\lambda)\).
\end{exa}

The existence of a locally compact quantum group structure on $\M$ is sufficient to construct Fourier maps \(\FT_{\Mh}\) and \(\FTh_{\M}\) satisfying \eqref{Fourier-structure}, i.e.
\begin{exa}[Locally compact quantum group]
\label{EX:locally_compact_quantum_group}
Let \(\M\) be a von Neumann algebra equipped with a unital normal \(*\)-homomorphism \(\coproduct\colon \M\to \M\otimes \M\) satisfying \((\coproduct \otimes \Id)\coproduct = (\Id\otimes \coproduct)\coproduct\), and with left- and right-invariant semi-finite faithful weights \(\tau_L\), \(\tau_R\) (the Haar weights), i.e.\
\(\tau_L((\omega\otimes \Id)\coproduct(x))=\tau_L(x)\omega(1)\) and \(\tau_R((\Id \otimes \omega)\coproduct(x))=\tau_R(x)\omega(1)\) for \(\omega\in \M^{+}_*\), \(x\in \mathcal{N}^{+}_{\tau_L}\) resp.\ \(\mathcal{N}^{+}_{\tau_R}\). This quadruplet \((\M,\coproduct,\tau_L,\tau_R)=:\G\) is a locally compact quantum group \cite{KV}. The dual \(\Gh\) is given by \(\widehat{\M}:=\{ (\omega\otimes \Id)(\mathbb{W}) : \omega\in B(\HH)_* \}\) \cite{KV}[Definition 1.6], where the multiplicative unitary \(\mathbb{W}\in B(\HH\otimes\HH)\) encodes the coproduct via \(\mathbb{W}^*(\Lambda(x)\otimes \Lambda(y)) = (\Lambda\otimes\Lambda)\coproduct(y)(x\otimes 1)\), \(x,y\in\mathcal{N}_{\tau_L}\). The Fourier and inverse Fourier maps are
\[
\FT(x)=(\tau_L \otimes \Id)(\mathbb{W}(x\otimes 1)),\quad x\in \mathcal{N}_{\tau}(\G);\qquad \FT^{-1}(y)=(\Id\otimes\widehat{\tau}_R)(\mathbb{W}^*(1 \otimes y)),\quad y\in \Gh.
\]
\end{exa}

\subsection{Hierarchy of examples}
\label{subsec:examples_summary}

The Fourier structure framework (Definition~\ref{Fourier-structure}) encompasses a wide range of classical and noncommutative settings. The existence of a Fourier structure is guaranteed for \emph{any} von Neumann algebra admitting a direct integral decomposition, by the theory of Lenz--Teplyaev \cite{LT16} and Sukochev's framework for affiliated operators \cite{dykema2016reduction}. The choice of underlying space manifests only through the Dirac-type operator $\D$ and its spectral dimension $\alpha$, which governs decay rates in $L^p$--$L^q$ estimates.

The following diagram illustrates the hierarchy of examples:
\begin{center}
\textbf{Von Neumann Algebras with Fourier Structure}\\[2mm]
\begin{tabular}{ccc}
\textsc{Commutative} & \textsc{Quantum/NC} & \textsc{Singular} \\[1mm]
$\downarrow$ & $\downarrow$ & $\downarrow$ \\[1mm]
Groups, Manifolds & Deformations & Fractals, Boundary \\[1mm]
$\swarrow \quad \searrow$ & $\swarrow \quad \searrow$ & $\swarrow \quad \searrow$ \\[1mm]
\footnotesize{Discrete \quad Continuous} & \footnotesize{$\mathbb{T}^n_\theta$ \quad $\mathbb{R}^n_q$} & \footnotesize{p.c.f. \quad Nonharmonic} \\[1mm]
\footnotesize{$(F_n$, Cayley)} & \footnotesize{$\downarrow_{q\to 1}$} & \footnotesize{(Sierpi\'nski)} \\[1mm]
$\downarrow$ & $\mathbb{R}^n$ & \\[1mm]
\footnotesize{Lie Groups} & $\uparrow$ & \\[1mm]
$\swarrow \quad \searrow$ & \footnotesize{$(k,a)$-Dunkl} & \\[1mm]
\footnotesize{Compact \quad Nilpotent} & \footnotesize{$(k\to 0)$} & \\[1mm]
\footnotesize{$(SU(n))$} \quad \footnotesize{(Stratified)} & & \\[1mm]
$\downarrow_{q\text{-deform}}$ & & \\[1mm]
\footnotesize{Quantum Groups} & & \\[1mm]
$\downarrow$ & & \\[1mm]
\footnotesize{Quantum Hypergroups} & &
\end{tabular}
\end{center}

\smallskip
\noindent\textbf{Choice of Dirac operator and spectral dimension.}
In each setting, the Dirac-type operator $\D$ is chosen canonically
from the geometry, and the spectral dimension $\alpha$ is determined
by Weyl-type asymptotics $\tau(E_{(0,s)}(|\D|)) \sim s^\alpha$:
\begin{itemize}
\item \emph{Classical $\mathbb{R}^n$}:
      $\D = (1-\Delta)^{1/2}$, $\alpha = n$ (topological dimension)
      \cite{Stein1970}.
\item \emph{Compact manifold $M^n$}:
      $\D = (1+\Delta_g)^{1/2}$, $\alpha = n$
      (Weyl's law \cite{Weyl1912,Shubin2001}).
\item \emph{Stratified groups}:
      $\D = (1+\mathcal{L})^{1/2}$ where $\mathcal{L}$ is the
      sub-Laplacian, $\alpha = Q$ (homogeneous dimension $Q > n$,
      \cite{FollandStein1982,VSC1992}).
\item \emph{Free groups $F_n$}:
      $\D = (I-\mu)^{-1/2}$ where $\mu$ is the Markov operator,
      $\alpha = 3$ \cite{Kesten1959,Woess2000}.
\item \emph{Fractals}:
      $\D = \sqrt{\Delta_K}$ via resistance forms,
      $\alpha = d_s \notin \mathbb{Z}$
      (non-integer spectral dimension,
      \cite{KigamiLapidus1993,FukushimaShima1992}).
\item \emph{Quantum tori $\mathbb{T}^n_\theta$}:
      $\D = (1+\Delta)^{1/2}$, $\alpha = n$
      (isospectral deformation \cite{ConnesLandi2001,RieffelYM2002}).
\item \emph{Quantum groups $\mathbb{G}_q$}:
      $\D = (1+L)^{1/2}$ where $L$ is the quantum Laplacian,
      $\alpha$ varies with $q$
      \cite{Drinfeld1986,Jimbo1985,VoigtYuncken2020}.
\item \emph{$(k,a)$-Dunkl}:
      $\D = (1+L_{k,a})^{1/2}$,
      $\alpha = n + 2\gamma_k$ where
      $\gamma_k = \sum_{\beta \in R^+} k(\beta)$
      \cite{Dunkl1989,Rosler1998,BenSaidKobayashiOrsted2012}.
\item \emph{Boundary-spectral}:
      $\D$ from eigenvalue asymptotics of the boundary problem
      \cite{Grubb1996}.
\end{itemize}

\smallskip
\noindent\textbf{Key relationships.}
Several examples are related by limits or deformations:
\begin{itemize}
\item $\mathbb{T}^n_\theta \xrightarrow{\theta \to 0} \mathbb{T}^n$
      (classical torus \cite{ConnesLandi2001});
\item $SU_q(n) \xrightarrow{q \to 1} SU(n)$
      (classical Lie group via Drinfeld--Jimbo quantization
      \cite{Drinfeld1986,Jimbo1985});
\item $(k,a)\text{-Dunkl} \xrightarrow{k \to 0} \mathbb{R}^n$
      (classical Fourier \cite{BenSaidKobayashiOrsted2012});
\item $\mathbb{R}^n \subset \text{Stratified}
      \subset \text{Homogeneous}
      \subset \text{Nilpotent Lie groups}$
      \cite{FollandStein1982,VSC1992}.
\end{itemize}

\begin{table}[h]
\centering
\small
\begin{tabular}{|l|c|l|}
\hline
\textbf{Setting} & $\alpha$ & \textbf{Source of $\alpha$} \\
\hline
$\mathbb{R}^n$
  & $n$
  & Topological dimension \cite{Stein1970} \\
Compact manifold $M^n$
  & $n$
  & Weyl's law \cite{Weyl1912,Shubin2001} \\
Heisenberg $\mathbb{H}^n$
  & $2n+2$
  & Homogeneous dimension $Q$ \cite{FollandStein1982} \\
Free group $F_n$
  & $3$
  & Kesten's theorem \cite{Kesten1959,Woess2000} \\
Sierpi\'nski gasket
  & $\log 9/\log 5$
  & Fractal spectral dimension
    \cite{KigamiLapidus1993,FukushimaShima1992} \\
Quantum torus $\mathbb{T}^n_\theta$
  & $n$
  & Isospectral to classical \cite{ConnesLandi2001} \\
$(k,a)$-Dunkl
  & $n + 2\gamma_k$
  & Root system contribution \cite{Dunkl1989,Rosler1998} \\
\hline
\end{tabular}
\caption{Spectral dimension $\alpha$ for key examples,
governing decay rates
$\|e^{-t\mathcal{L}}\|_{L^p \to L^q}
 \lesssim t^{-\alpha(1/p - 1/q)/2}$.}
\label{tab:examples}
\end{table}


\section{Paley and Hausdorff-Young-Paley type inequalities on von Neumann algebras endowed with Fourier structure}
\label{sec:paley}
We start this section with the Hausdorff-Young and Paley-type inequalities. For the results in this section, we shall assume that von Neumann algebras $\M$ can be equipped with a Fourier structure defined in Definition \ref{Fourier-structure}.
\begin{thm}\label{THM:H-Y-ineq} (Hausdorff-Young inequality)\label{H-Y} Let $1\leq p\leq2$ with $\frac{1}{p}+\frac{1}{p'}=1$ and let $(\M,\Mh)$ be a tuple of general von Neumann algebras admitting a Fourier structure. Then for any $x\in L^{p}(\M)$ we have
 \begin{eqnarray}\label{H-Y_ineq}
\|\FT_{\Mh}[x]\|_{L^{p'}(\Mh)} \leq \|x\|_{L^{p}(\M)}.
\end{eqnarray}
 \end{thm}
\begin{proof}
By formulas \eqref{Plancherel} and \eqref{FT-l1-l-infty} in Definition \ref{Fourier-structure}, the Fourier map $\FT_{\Mh}$ is well defined and extends to a bounded map from $L^{1}(\M)$ to $L^{\infty}(\Mh)$, and by the Plancherel identity, we have
\begin{equation*}
\|\FT_{\Mh}[x]\|_{L^{2}(\Mh)}=\|x\|_{L^{2}(\M)},\quad x\in L^{2}(\M).
\end{equation*}
In other words, $\FT_{\Mh}$ is bounded from $L^{1}(\M)$ to $L^{\infty}(\Mh)$ and from $L^{2}(\M)$ to $L^{2}(\Mh).$ Hence, the assertion follows from the noncommutative complex interpolation (see, for example \cite{PXu}).
\end{proof}
If we use real interpolation in Theorem~\ref{H-Y} instead of complex interpolation, then we obtain the following version of the Hausdorff--Young inequality.
\begin{lem}\label{H-Y-Lorentz} Let $1\leq p\leq2$ with $\frac{1}{p}+\frac{1}{p'}=1$, let $1\leq q\leq\infty$, and let $(\M,\Mh)$ be a tuple of semi-finite von Neumann algebras admitting a Fourier structure.  Then for any $x\in L^{p,q}(\M)$ we have
 \begin{eqnarray}\label{R-H-Y_ineq}
\|\FT_{\Mh}[x]\|_{L^{p',q}(\Mh)} \lesssim\|x\|_{L^{p,q}(\M)}.
\end{eqnarray}
 \end{lem}

The following Paley-type inequality follows directly from Lemma \ref{H-Y-Lorentz} and Proposition \ref{Holder-ineq-Lorentz-spaces}
\begin{thm}\label{Paley-ineq} Let $\M$ be a semi-finite von Neumann algebra and let  $1<p\leq2.$ If $\varphi$ is a positive function on $\mathbb{R}_{+}$  satisfying condition
\begin{equation*}
M_{\varphi}:=\sup _{s>0} s \int\limits_{\substack{t \in \mathbb{R}_{+} \\ \varphi(t) \geq s}} d t<\infty,
\end{equation*}
then for all $x \in L^{p}(\M)$ we have
\begin{equation}\label{P_ineq}
\left(\int_{\mathbb{R}_{+}}\mu^p(t;\FT_{\Mh}[x]) \varphi(t)^{2-p}d t\right)^{\frac{1}{p}} \lesssim M_{\varphi}^{\frac{2-p}{p}}\|x\|_{L^{p}(\M)}.
\end{equation}
\end{thm}
\begin{proof} Let $\frac{1}{r}:=\frac{2}{p}-1.$ Then we have $\frac{1}{p}=\frac{1}{p'} + \frac{1}{r}.$ Hence, applying inequality \eqref{Lorentz-holder} in Proposition \ref{Holder-ineq-Lorentz-spaces} and \eqref{R-H-Y_ineq} in Lemma \ref{H-Y-Lorentz} for any $x\in L^p(\M)$ we obtain
\begin{eqnarray*}\|\mu(\FT_{\Mh}[x])\cdot\varphi(\cdot)^{\frac{2-p}{p}}\|_{L^{p}(\mathbb{R}_+)}&=&\|\mu(\FT_{\Mh}[x])\cdot\varphi(\cdot)^{\frac{2-p}{p}}\|_{L^{p,p}(\mathbb{R}_+)}\\
&\overset{\eqref{Lorentz-holder}}{\lesssim}&\|\varphi(\cdot)^{\frac{2-p}{p}}\|_{L^{r,\infty}(\mathbb{R}_+)}\|\mu(\FT_{\Mh}[x])\|_{L^{p',p}(\mathbb{R}_+)}\\
&\overset{\eqref{NC Lorentz space norm II}}{=}&\|\varphi(\cdot)^{\frac{2-p}{p}}\|_{L^{r,\infty}(\mathbb{R}_+)}\|\FT_{\Mh}[x]\|_{L^{p',p}(\Mh)}\\
&\overset{\eqref{R-H-Y_ineq}}{\lesssim}&\|\varphi(\cdot)^{\frac{2-p}{p}}\|_{L^{r,\infty}(\mathbb{R}_+)}\|x\|_{L^{p}(\M)}.
\end{eqnarray*}
Here, we used the fact that
$$\|\varphi(\cdot)^{\frac{2-p}{p}}\|^r_{L^{r,\infty}(\mathbb{R}_+)}=M_{\varphi}<\infty.$$
Indeed, by an easy calculation we have
\begin{eqnarray*}\|\varphi(\cdot)^{\frac{2-p}{p}}\|^r_{L^{r,\infty}(\mathbb{R}_+)}&\overset{\eqref{weak-quasi-norm}}{=}&\sup\limits_{s>0}s^r m\{\xi>0:\varphi(\xi)^{\frac{2-p}{p}}>s\}\\
&=&\sup\limits_{s>0}s^r m\{\xi>0:\varphi(\xi)>s^{\frac{p}{2-p}}\}\\
&=&\sup\limits_{s>0}s^r m\{\xi>0:\varphi(\xi)>s^{r}\}\\
&=&\sup\limits_{s>0}s m\{\xi>0:\varphi(\xi)>s\}\\
&=&\sup\limits_{s>0}s \int\limits_{\varphi(\xi)>s}d\xi=M_{\varphi}<\infty.
\end{eqnarray*}
This concludes the proof.
\end{proof}
Further, we recall a result on the interpolation of weighted spaces which can be found from \cite[Theorem 5.5.1 and Corollary 5.5.2, pp. 119-120]{BL1976}.
\begin{thm}\label{Thm 3.2} (Interpolation of weighted spaces). Let $ d\nu_0(\xi) = \omega_0(\xi)d\xi$, $d\nu_1(\xi)=\omega_1(\xi)d\xi$, and write $L^p(\omega)=L^p(\omega d\nu)$ for the weight $\omega.$ Suppose that $0<  p_0, p_1 < \infty.$ Then we have $$\left(L^{p_0}(\omega_0),L^{p_1}(\omega_1)\right)_{\theta,p}=  L^p(\omega),$$
where $0 < \theta < 1,$ $\frac{1}{p}=\frac{1-\theta }{p_0}+\frac{\theta}{p_1},$ and $\omega=\omega_0^\frac{p(1-\theta)}{p_0}\cdot\omega_1^\frac{p\theta }{p_1}$.
\end{thm}
From this, interpolating between the Hausdorff-Young inequality \eqref{H-Y_ineq} and Paley-type inequality \eqref{P_ineq} in Theorem \ref{Paley-ineq}, we readily obtain the following inequality.

\begin{thm}\label{Thm 3.3}(Hausdorff-Young-Paley inequality). Let $1<p \leq q \leq p^{\prime}<\infty$ with $\frac{1}{p}+\frac{1}{p'}=1.$ If $\varphi$ is a positive function on $\mathbb{R}_{+}$  satisfying condition
\begin{equation*}
M_{\varphi}:=\sup _{s>0} s \int\limits_{\substack{t \in \mathbb{R}_{+} \\ \varphi(t) \geq s}} d t<\infty,
\end{equation*}
then for all $x \in L^{p}(\M)$ we have
\begin{equation}\label{H-Y-P-ineq}
\left(\int_{\mathbb{R}_{+}}\left(\mu(t;\FT_{\Mh}[x]) \varphi(t)^{\frac{1}{q}-\frac{1}{p^{\prime}}}\right)^{q} d t\right)^{\frac{1}{q}} \lesssim M_{\varphi}^{\frac{1}{q}-\frac{1}{p^{\prime}}}\|x\|_{L^{p}(\M)}.
\end{equation}
\end{thm}
Naturally, this reduces to the Hausdorff-Young inequality \eqref{H-Y_ineq} when $q=p^{\prime}$ and to the Paley-type inequality in \eqref{P_ineq} when $q=p.$

\section{\texorpdfstring{Hardy-Littlewood-type inequalities associated with $\D$}{Hardy-Littlewood-type inequalities associated with D}}
\label{sec:hardy-littlewood}

As an immediate consequence of Theorem \ref{Paley-ineq}, we obtain a differential formulation of the Hardy-Littlewood inequality. For the rest of this section we shall assume that a linear operator $\D\colon L^2(\M) \to L^2(\M)$ is affiliated with $\Mh$ and its absolute part has bounded inverse $|\D|^{-1}$ which is measurable with respect to $\Mh$, i.e. $|\D|^{-1}\in L^0(\M)$.
Moreover, we shall assume that for some $\beta> 0$, the inverse $|\D|^{-1}$ is integrable in a suitable $L^p$-based space.

For singular traces $\tau_{\omega}$ on general von Neumann algebras, the smallest $\beta$ for which $\|\lvert\D\rvert^{-\beta}\|_{L^{1,\infty}(\Mh,\tau_{\omega})} < \infty$ can be interpreted as the metric dimension associated with a spectral triple structure on $\M$.

Given a pair of von Neumann algebras $(\M,\Mh)$, we shall use a pair of operators $\D$ and $\widehat{\D}$ where the former operator is affiliated with $\Mh$ and the latter is affiliated with $\M$, i.e.
\begin{eqnarray}
\D &\in& L^0(\Mh), \nonumber \\
\widehat{\D} &\in& L^0(\M).
\end{eqnarray}
\begin{thm}\label{THM_HL_Dirac} Let $1<p \leq 2$ and let $(\M,\Mh)$ be a tuple of  von Neumann algebras admitting a Fourier structure.
Then for all $x \in L^{p}(\M)$ we have
\begin{equation}\label{EQ_HL_Dirac}
\left\|\FT_{\Mh}\left(\D^{-\beta} x\right)\right\|_{L^{p}(\Mh)} \lesssim\left\|\D^{-\beta}\right\|_{L^{r}(\Mh)}\|x\|_{L^{p}(\M)}, \quad \frac{1}{r}=\frac{2-p}{p}.
\end{equation}
Here, we choose such $\beta>0$ that the norm $\left\|\D^{-\beta}\right\|_{L^{r}(\Mh)}$ is finite.
\end{thm}

\begin{proof}[Proof of Theorem \ref{THM_HL_Dirac}]
Since the operator $\D$ is affiliated with $\Mh$ and $\FT_{\Mh}$ is the Fourier transform, we get from the property \eqref{eq:affiliation} of the Fourier structure (see Definition~\ref{Fourier-structure})
\begin{equation}
\label{THM:Dirac-HL-EQ-1}
\FT_{\Mh}(\D^{-\beta}x)
=\D^{-\beta}\FT_{\Mh}(x).
\end{equation}
By H\"older inequality \eqref{EQ:holder-inequality}, we get
\begin{equation}
\label{THM:Dirac-HL-EQ-2}
\| \D^{-\beta} \FT_{\Mh}(x)\|_{L^p(\Mh)}
\leq
\| \D^{-\beta}\|_{L^{r}(\Mh)}
\| \FT_{\Mh}(x)\|_{L^{p'}(\Mh)}
\leq \\
\| \D^{-\beta}\|_{L^{r}(\Mh)}
\| x\|_{L^{p}(\M)},
\end{equation}
where in the last inequality we used Hausdorff-Young inequality in the form of Theorem \ref{THM:H-Y-ineq}.
Collecting \eqref{THM:Dirac-HL-EQ-1} and  \eqref{THM:Dirac-HL-EQ-2}, we obtain \eqref{EQ_HL_Dirac}.
This concludes the proof.
\end{proof}
In the case of semi-finite von Neumann algebras, Theorem \ref{THM_HL_Dirac} admits a sharper form.
\begin{thm}\label{THM_HL_Dirac_semifinite}
Let $1<p\le 2$ and define
\[
\frac{1}{r}=\frac{2}{p}-1,
\qquad
\frac{1}{s}=\frac{1}{s_1}+\frac{1}{s_2},
\qquad
0<s,s_1,s_2\le\infty .
\]
Let $(\M,\widehat{\M})$ be a pair of semi-finite von Neumann algebras
admitting a Fourier structure.
Then for every $x\in L^{p}(\M)$ one has
\[
\|\FT_{\widehat{\M}}(\D^{-\beta}x)\|_{L^{p,s}(\widehat{\M})}
\lesssim
\|\lvert \D\rvert^{-\beta}\|_{L^{r,s_1}(\widehat{\M})}
\|x\|_{L^{p,s_2}(\M)} .
\]
\end{thm}
\begin{proof}[Proof of Theorem \ref{THM_HL_Dirac_semifinite}]
    Since the operator $\D$ is affiliated with $\Mh$ and $\FT_{\Mh}$ denotes the
    Fourier transform, property \eqref{eq:affiliation} of the Fourier structure
    (see Definition~\ref{Fourier-structure}) implies
    \begin{equation}\label{THM:Dirac-HL-EQ-1-semifinite}
    \FT_{\Mh}(\D^{-\beta}x)=\D^{-\beta}\FT_{\Mh}(x).
    \end{equation}
    
    Applying the H\"older inequality for Lorentz spaces
    \eqref{Lorentz-holder}, we obtain
    \begin{equation}\label{THM:Dirac-HL-EQ-2-semifinite}
    \|\D^{-\beta}\FT_{\Mh}(x)\|_{L^{p,s}(\Mh)}
    \le
    \|\D^{-\beta}\|_{L^{r,s_1}(\Mh)}
    \|\FT_{\Mh}(x)\|_{L^{p',s_2}(\Mh)} .
    \end{equation}
    
    Next, the Hausdorff--Young inequality (Theorem \ref{THM:H-Y-ineq})
    yields
    \[
    \|\FT_{\Mh}(x)\|_{L^{p',s_2}(\Mh)}
    \lesssim
    \|x\|_{L^{p,s_2}(\M)} .
    \]
    
    Combining this estimate with
    \eqref{THM:Dirac-HL-EQ-1-semifinite} and
    \eqref{THM:Dirac-HL-EQ-2-semifinite} gives
    \begin{align*}
        \|\FT_{\Mh}(\D^{-\beta}x)\|_{L^{p,s}(\Mh)}
        &\lesssim
        \|\D^{-\beta}\|_{L^{r,s_1}(\Mh)} \\
        &\quad \|x\|_{L^{p,s_2}(\M)} .
        \end{align*}
    This proves the theorem.
    \end{proof}
The following is the dual counterpart of Theorem~\ref{THM_HL_Dirac_semifinite},
obtained by applying it to the pair $(\widehat{\M}, \M)$ in place of $(\M, \widehat{\M})$.

\begin{thm}\label{THM:dual-HLP}
Let $1<p\le 2$
and define
\[
\frac{1}{r}
=
\frac{2}{p}-1,
\qquad
\frac{1}{s}
=
\frac{1}{s_1}+\frac{1}{s_2},
\qquad
0<s,s_1,s_2\le\infty.
\]
Let $(\Mh,\M)$ be a tuple of von Neumann algebras
admitting a Fourier structure.
Then we have
\begin{equation}
\label{EQ:dual-hardy-littlewood-paley-lorentz}
\|x\|_{L^{p,s}(\M)}
\lesssim
\|\Dh^{-\beta}\|_{L^{r,\, s_1}(\M)}
\,
\|\D^{\beta}\FTh_{\Mh}[x]\|_{L^{p,\, s_2}(\Mh)}.
\end{equation}
\end{thm}

\begin{proof}[Proof of Theorem \ref{THM:dual-HLP}]
By Definition~\ref{Fourier-structure} \eqref{eq:fourier1},
\begin{equation}
\label{EQ:dual-HLP-by-def}
x
=
\FTh_{\M}[\FT_{\Mh}(x)].
\end{equation}
By Theorem~\ref{H-Y}, we immediately get
\begin{equation}
\label{EQ:dual-HLP-inverse-HY}
\|\FTh_{\M}[\FT_{\Mh}(x)]\|_{L^{p,s_1}(\M)}
\leq
\|\FT_{\Mh}(x)\|_{L^{p',s_1}(\M)}.
\end{equation}
By the H\"older inequality for Lorentz spaces
\eqref{Lorentz-holder}
in Proposition~\ref{Holder-ineq-Lorentz-spaces},
with $\frac{1}{p'}=\frac{1}{r}+\frac{1}{p}$
and $\frac{1}{s_2}=\frac{1}{s_1}+\frac{1}{s}$,
\begin{equation}
\label{EQ:dual-HLP-Holder}
\|\FT_{\Mh}(x)\|_{L^{p',s_2}(\Mh)}
\le
\|\lvert\D\rvert^{-\beta}\|_{L^{r,s_1}(\Mh)}
\,
\|\D^{\beta}\FT_{\Mh}(x)\|_{L^{p,s_2}(\Mh)}.
\end{equation}
By the convention before Theorem~\ref{THM_HL_Dirac},
$\D$ is affiliated with $\Mh$
and $\Dh$ with $\M$,
so
\begin{equation}
\label{EQ:dual-HLP-convention}
\|\lvert\D\rvert^{-\beta}\|_{L^{r,s_1}(\Mh)}
=
\|\Dh^{-\beta}\|_{L^{r,s_1}(\M)}.
\end{equation}
Combining
\eqref{EQ:dual-HLP-inverse-HY},
\eqref{EQ:dual-HLP-Holder},
and
\eqref{EQ:dual-HLP-convention},
\begin{equation}
\label{EQ:dual-HLP-combined}
\|x\|_{L^{p,s}(\M)}
\lesssim
\|\Dh^{-\beta}\|_{L^{r,s_1}(\M)}
\,
\|\D^{\beta}\FT_{\Mh}(x)\|_{L^{p,s_2}(\Mh)}.
\end{equation}
Thus, we established \eqref{EQ:dual-hardy-littlewood-paley-lorentz}.
This completes the proof.
\end{proof}

\section{H\"ormander-Mihlin type multiplier theorems, \texorpdfstring{$p\leq q$}{p <= q}}
\label{sec:mihlin}
In this section, we establish the two main results of the paper.

\begin{thm}\label{THM:Mihlin-Hormander-semifinite-p-2-q}
Let $1 < p \leq 2 \leq q < \infty$, let $(\M,\Mh)$ admit a Fourier structure,
and let $\D, A \colon L^2(\M) \to L^2(\M)$ be linear operators affiliated with $\Mh$.
Then for every $x \in L^{p}(\M)$,
\begin{equation}
\label{EQ:MH-Lp-Lq}
\|Ax\|_{L^{q}(\M)}
\lesssim
\|A\|_{L^{r}(\Mh)}
\|x\|_{L^{p}(\M)},
\qquad
\frac{1}{r} = \frac{1}{p} - \frac{1}{q}.
\end{equation}
\end{thm}
\begin{proof}[Proof of Theorem \ref{THM:Mihlin-Hormander-semifinite-p-2-q}]
Since $A$ is affiliated with $\Mh$, property~\eqref{eq:affiliation} of the
Fourier structure gives $\FT_{\Mh}[Af] = A\,\FT_{\Mh}[f]$.
Using this together with the Hausdorff--Young inequality~\eqref{R-H-Y_ineq}
and the H\"older inequality for Lorentz spaces~\eqref{Lorentz-holder}, we estimate
\begin{align}
\|Af\|_{L^q(\M)}
&\overset{\phantom{\eqref{Lorentz-holder}}}{\leq}
 \|Af\|_{L^{q,\,p}(\M)}
 \overset{\eqref{R-H-Y_ineq}}{\leq}
 \bigl\|\FT_{\Mh}[Af]\bigr\|_{L^{q',p}(\Mh)}
\label{EQ:MH-pq-step1}\\
&\overset{\eqref{eq:affiliation}}{=}
 \bigl\|A\,\FT_{\Mh}[f]\bigr\|_{L^{q',p}(\Mh)}
 \overset{\eqref{Lorentz-holder}}{\leq}
 \|A\|_{L^{r,\infty}(\Mh)}
 \bigl\|\FT_{\Mh}[f]\bigr\|_{L^{p',p}(\Mh)}
\label{EQ:MH-pq-step2}\\
&\overset{\eqref{R-H-Y_ineq}}{\lesssim}
 \|A\|_{L^{r,\infty}(\Mh)}\,\|f\|_{L^{p}(\M)}.
\label{EQ:MH-pq-step3}
\end{align}
This gives~\eqref{EQ:MH-Lp-Lq} and completes the proof.
\end{proof}

\begin{thm}
\label{THM:Mihlin-Hormander-general}
Let $(\M,\Mh)$ be a tuple of  von Neumann algebras admitting a Fourier structure.  Let $\widehat{\D} \colon L^2(\M) \to L^2(\M)$ be a linear operator affiliated with $\M$ and let $A\colon L^2(\M) \to L^2(\M)$ be a linear operator affiliated with $\Mh$.
Then for every $x \in L^{p}(\M)$ we have
\begin{equation}
\label{EQ:MH-Lp-Lp}
\|A x\|_{L^{p}(\M)}
\lesssim
\left\|\Dh^{-\beta}\right\|_{L^{\infty}(\Mh)}
\left\|\Dh^{\beta} A
\right\|_{L^{r}(\Mh)}\|x\|_{L^{p}(\M)},
\quad
\frac{1}{r}=2\left|\frac{1}{p}-\frac{1}{2}\right|.
\end{equation}
\end{thm}
\begin{proof}[Proof of Theorem \ref{THM:Mihlin-Hormander-general}]
 Let us first assume $1<p \leq 2$. By Theorem \ref{THM:dual-HLP}, we get
\begin{equation}
\label{PROOF:EQ:MH-step-1}
\|A x\|_{L^{p}(\M)}
\lesssim
\left\|\Dh^{-\beta}\right\|_{L^{\infty}(\Mh)}
\left\|\Dh^{\beta} \FT_{\Mh}[A x]\right\|_{L^{p}(\Mh)}.
\end{equation}
Since $A$ is affiliated with $\Mh$, we apply  property \eqref{eq:affiliation}, we get
\begin{equation}
\label{PROOF:EQ:MH-step-2}
\FT_{\Mh}[A x] = A \FT_{\Mh}[x]
\end{equation}
Applying \eqref{PROOF:EQ:MH-step-2}  and \eqref{EQ:holder-inequality} of Proposition~\ref{PROP:izumi_holder_inequality} for $L^p$ space, we get
\begin{equation}
\label{PROOF:EQ:MH-step-3}
\left\|\Dh^{\beta} \FT_{\Mh}[A x]\right\|_{L^{p}(\Mh)}
=
\left\|\Dh^{\beta} A \FT_{\Mh}[ x]\right\|_{L^{p}(\Mh)}
\leq
\|\Dh^{\beta}A\|_{L^r(\Mh)}
\| \FT_{\M}[x]\|_{L^{p'}(\Mh)}.
\end{equation}
By Hausdorff-Young inequality of Theorem \ref{THM:H-Y-ineq}, we get
\begin{equation}
\label{PROOF:EQ:MH-step-4}
\| \FT_{\M}[x]\|_{L^{p'}(\Mh)} \leq \|x\|_{L^p(\M)}.
\end{equation}
Collecting \eqref{PROOF:EQ:MH-step-1},\eqref{PROOF:EQ:MH-step-2}, \eqref{PROOF:EQ:MH-step-3}, \eqref{PROOF:EQ:MH-step-4}, we obtain \eqref{EQ:MH-Lp-Lp} for the range $1<p \leq 2$.
By the $L^p$-space duality, the same inequality holds true for $2\leq p < \infty$. This completes 
the proof.
\end{proof}
Now, we assume additionally that $\M$ is a semi-finite von Neumann algebra. Then we can obtain a more refined version of Theorem  \ref{THM:Mihlin-Hormander-general}:
\begin{thm}
\label{THM:Mihlin-Hormander-semifinite}
Let $(\M,\Mh)$ be a tuple of  semi-finite von Neumann algebras admitting a Fourier structure.  Let $\Dh \colon L^2(\M) \to L^2(\M)$ be a linear operator affiliated with $\M$ and let $A\colon L^2(\M)\to L^2(\M)$ be a linear operator affiliated with $\Mh$.
Then for every $x \in L^{p}(\M)$ we have
\begin{equation}
\label{EQ:MH-Lp-Lp-semifinite}
\|A x\|_{L^{p, s}(\M)}
\lesssim
\|\Dh^{-\beta}\|_{L^{r,s_1}(\Mh)}
\left\|\Dh^{\beta} A
\right\|_{L^{r, s_3}(\Mh)}\|x\|_{L^{p,\,s_4}(\M)},
\end{equation}
where
$$
\frac{1}{r}=2\left|\frac{1}{p}-\frac{1}{2}\right|, \quad  \frac1s=\frac{1}{s_1} + \frac1s_3 + \frac1{s_4},\quad 0 < s_1, s_3, s_4 \leq \infty.
$$
\end{thm}

\begin{rem} The operator $\Dh$ affiliated with $\M$ and composed with $A$ is measuring frequency space smoothness of $A$.
\end{rem}
\begin{proof}[Proof of Theorem \ref{THM:Mihlin-Hormander-semifinite}]

Let us first assume $1 < p \leq 2$. Since $A$ is affiliated with $\M$,
property~\eqref{eq:affiliation} of the Fourier structure gives
\begin{equation}
\label{EQ:MH-sf-step1}
\D^{\beta}\FT_{\Mh}[Ax]
=
\D^{\beta}A\,\FT_{\Mh}[x],
\qquad
x \in L^p(\M).
\end{equation}
\medskip

Applying Theorem~\ref{THM:dual-HLP}
with $\frac{1}{s}=\frac{1}{s_1}+\frac{1}{s_2}$ and $\frac1r=\frac{2}{p}-1$,
\begin{equation}
\label{EQ:MH-sf-step2}
\|Ax\|_{L^{p,s}(\M)}
\lesssim
\|\Dh^{-\beta}\|_{L^{r,s_1}(\Mh)}
\,
\|\D^{\beta}\FT_{\Mh}[Ax]\|_{L^{p,s_2}(\Mh)}.
\end{equation}
\medskip
Using \eqref{EQ:MH-sf-step1}
and applying \eqref{Lorentz-holder}
with $\frac{1}{s_2}=\frac{1}{s_3}+\frac{1}{s_4}$,
$0 < s_2, s_3, s_4 \leq \infty$,
\begin{equation}
\label{EQ:MH-sf-step3}
\|\D^{\beta}A\,\FT_{\Mh}[x]\|_{L^{p,s_2}(\Mh)}
\leq
\|\D^{\beta}A\|_{L^{r,s_3}(\Mh)}
\,
\|\FT_{\Mh}[x]\|_{L^{p',s_4}(\Mh)}.
\end{equation}
\medskip
Applying \eqref{R-H-Y_ineq},
\begin{equation}
\label{EQ:MH-sf-step4}
\|\FT_{\Mh}[x]\|_{L^{p',s_4}(\Mh)}
\lesssim
\|x\|_{L^{p,s_4}(\M)}.
\end{equation}
\medskip
Collecting
\eqref{EQ:MH-sf-step2},
\eqref{EQ:MH-sf-step3},
and
\eqref{EQ:MH-sf-step4},
we obtain
\begin{equation}
\label{EQ:MH-sf-chain}
\begin{aligned}
\|Ax\|_{L^{p,s}(\M)}
&\overset{\eqref{EQ:MH-sf-step2}}{\lesssim}
\|\Dh^{-\beta}\|_{L^{r,s_1}(\Mh)}
\,
\|\D^{\beta}\FT_{\Mh}[Ax]\|_{L^{p,s_2}(\Mh)}
\\[4pt]
&\overset{\eqref{EQ:MH-sf-step1}}{=}
\|\Dh^{-\beta}\|_{L^{r,s_1}(\Mh)}
\,
\|\D^{\beta}A\,\FT_{\Mh}[x]\|_{L^{p,s_2}(\Mh)}
\\[4pt]
&\overset{\eqref{EQ:MH-sf-step3}}{\leq}
\|\Dh^{-\beta}\|_{L^{r,s_1}(\Mh)}
\,
\|\D^{\beta}A\|_{L^{r,s_3}(\Mh)}
\,
\|\FT_{\Mh}[x]\|_{L^{p',s_4}(\Mh)}
\\[4pt]
&\overset{\eqref{EQ:MH-sf-step4}}{\lesssim}
\|\Dh^{-\beta}\|_{L^{r,s_1}(\Mh)}
\,
\|\D^{\beta}A\|_{L^{r,s_3}(\Mh)}
\,
\|x\|_{L^{p,s_4}(\M)}.
\end{aligned}
\end{equation}
By $L^p$-duality,
the same inequality holds for $2 \leq p < \infty$.
This completes the proof.
\end{proof}

\begin{exa}[Integrability criterion for $\D^{-\beta}$]
\label{Exa_integrability_criterion}
Let $(\M,\tau)$ be a semi-finite von Neumann algebra and let $\D$ be a positive self-adjoint operator affiliated with $\M$ satisfying the spectral asymptotics
$$
\tau\left(E_{(0, s)}(|\D|)\right) \cong s^{\alpha}, \quad s>1,
$$
for some $\alpha > 0$. Then we have
$$
\left\|\D^{-\beta}\right\|_{L^{r,\infty}(\M)}<\infty
$$
if and only if
$$
\frac{1}{r}<\frac{\beta}{\alpha}.
$$
\end{exa}

\section{Noncommutative Littlewood-Paley theory}
\label{SEC_noncommutative_littlewood_paley}
The classical theory of von Neumann algebras \cite{murray1936rings,murray1937rings,neumann1940rings,murray1943rings,von1949rings, dixmier2011neumann} allows us to decompose any von Neumann algebra into a measurable field of von Neumann algebras.
The results obtained in \cite{dykema2016reduction} develop the theory of direct integral decompositions for both bounded and unbounded operators, including foundational results on spectral projections, functional calculus, and operator affiliation with von Neumann algebras. In particular, it is shown that operators (either bounded or affiliated) with a semi-finite (tracial) von Neumann algebra admit a direct integral decomposition.

In the Euclidean setting, the sharp H\"ormander-type condition (Theorem~\ref{thm_grafakos}) is formulated using a Littlewood--Paley decomposition: localisation of the symbol to dyadic annuli via $\widehat{\Psi}(2^{-j}\xi)$ and a supremum over $j$ of (weighted) norms, which allows one to verify multiplier boundedness by checking symbolic conditions instead of a single global operator norm.
To obtain a noncommutative analogue of this formulation and to recover the classical result on $\mathbb{R}^n$ (Example~\ref{EX_case_of_Rn}), we develop a Littlewood--Paley framework in the setting of von Neumann algebras.
We use the reduction theory summarised above, together with the Fourier structure (Section~2) and the multiplier theorems of Sections~5--6, to express $L^p$-boundedness in terms of fibre-wise (symbolic) conditions and to prove Theorem~\ref{THM:Mihlin-Hormander-general-noncomm-L-Paley}.


Let $(X,\mu)$ be a standard Borel measure space with positive measure $\mu$ and let us denote by $L^{\infty}_C(X,\mu)$ the set of essentially bounded measurable complex-valued functions on $(X,\mu)$.
For $f\in L^{\infty}_C(X,\mu)$ let us denote by $T_f$ the operator of multiplication by $f$ acting on $L^2(X,\mu)$.
The operators $A$ of the form $T_f$, i.e. $A=T_f$ for some $f\in L^{\infty}_C(X,\mu)$ are said to be diagonalisable and we denote by $Z$ the abelian algebra generated by operators $T_f$.
The mapping $L^{\infty}_C(X,\mu) \ni f \mapsto T_f $ is a canonical homomorphism of $*$-involutive algebras \cite[Proposition~3, p.~182]{dixmier2011neumann}.

A von Neumann algebra $\M$ is said to be decomposable if it is defined by a measurable field of von Neumann algebras.

Let us fix an abelian subalgebra $Z$ of $\M$, i.e. $Z\subset \M$.
\begin{prop}[\cite{dykema2016reduction}]
\label{PROP:reduction}
Let $\M$ be a von Neumann algebra and let $\omega$ be a normal semi-finite faithful weight on $\M$ and let us denote by $Z$ an abelian subalgebra of $\M$. Then there exists a measure space $(X,\mu)$ with a positive bounded measure $\mu$ such that the following holds
\begin{enumerate}
\item The algebra $\M$ can be decomposed into a measurable field of von Neumann algebras $\M^{\lambda}$
\begin{equation}
\label{EQ_central_decomposition}
\M = \bigoplus_{X}\int \M^{\lambda} d\mu(\lambda)
\end{equation}
\item The weight $\omega$ can be decomposed into a measurable field of weights $\omega^{\lambda}$
\begin{equation}
\label{EQ_weight_decomposition}
\omega = \bigoplus_{X} \int \omega^{\lambda} d\mu(\lambda)
\end{equation}
\end{enumerate}
\end{prop}

One can estimate the Lorentz norm of $A \in L^0(\M)$ via fiber norms of the fiber operators $A(\lambda)$. We now show how to localize and measure the ``smoothness'' of $A$ with respect to $\Dh$.
\begin{definition}
Let $(X,\Sigma,\mu)$ be a $\sigma$-finite measure space, and let 
$T:X\to X$ be a measurable transformation. 
Denote by $\nu := \mu \circ T^{-1}$ the pushforward measure. 
If $\nu \ll \mu$, then by the Radon--Nikodým theorem there exists a density
\[
J_T := \frac{d\nu}{d\mu} \in L^1_{\mathrm{loc}}(\mu),
\]
called the \emph{Radon--Nikodým derivative} of $T$.
We say that $T$ has a \emph{bounded Radon--Nikodým derivative} if there exists
a constant $C>0$ such that
\[
J_T(x) \leq C \quad \text{for $\mu$-a.e.\ $x \in X$}.
\]
Equivalently,
\[
\mu\!\bigl(T^{-1}(A)\bigr) \leq C\,\mu(A), \qquad \forall A \in \Sigma.
\]
\end{definition}

\begin{definition}
Suppose there exists a family of measurable subsets 
$\{E_j\}_{j\in\mathbb{Z}} \subset \Sigma$ of finite measures $\mu(E_j) <+\infty$ and measurable maps 
$T_j:X\to X$, each with uniformly bounded Radon--Nikodým derivatives,
such that
\[
E_i \cap E_{i+1} \subset E_{i+1}, \qquad 
E_i \cap E_{i+2} = \emptyset, \qquad 
T_j^{-1}(E_j) = E_0 \quad \text{for all $j\in\mathbb{Z}$}.
\]
Let us fix a measurable positive function $\Psi$ with support in $E_0$. We introduce the dilations
$$
\Psi_j(x)=\Psi(T^{-1}_j(x)),\quad x \in X
$$
Now, let us define the bounded linear operator $\widehat{\Psi_j} = \bigoplus\limits_{X}\int \Psi_j(x)d\mu(x)$.

\end{definition}
It can be seen that  $X$ admits a decomposition into the disjoint union of the 
even--indexed sets $E_{2k}$:
\[
X = \bigsqcup_{k\in\mathbb{Z}} E_{2k}.
\]
The tuple $(\{(E_j\, T_j)\}_{j\in\mathbb{Z}}, \Psi)$ consisting of sequence $\{(E_j\, T_j)\}_{j\in\mathbb{Z}}$ of pairs $(E_j, T_j)$ and $\Psi$ we shall call Littlewood-Paley-decomposition over $X$.
\begin{definition}[$\Mh$-smoothness]
Let $\Dh$ be a linear operator affiliated with a von Neumann algebra $\M$ and let $A$ be a linear operator affiliated with $\Mh$. Then we shall say that $A$ is $\Dh^{\beta}$ regular if the composition $\Dh^{\beta} A$ is a linear operator affiliated with $\Mh$.
\end{definition}
\begin{prop}
\label{PROP:compute_on_center}
Let $\Dh\in L^0(\M)$ and $A\in L^0(\Mh)$ and let us  fix an abelian subalgebra $Z$ of $\Mh$. 
Let us assume that $\Dh$ can be decomposed into measurable field $\Dh(\lambda)$ of operators over the spectrum $\widehat{Z}$ of $Z$, i.e.
\begin{equation}
\label{EQ_Dh_symbol}
\Dh_j = \bigoplus_{\widehat{Z}}\int \limits \LL(T^{-1}_j\lambda) d\mu(\lambda).
\end{equation}
Let us fix a Littlewood-Paley decomposition $\{E_j, T_j, \Psi_j\}$ over the spectrum $\widehat{Z}$ of $Z$.
Then we have 
\begin{equation}
\label{EQ:compute_on_center}
\tau(E_{(s,+\infty)}(|\Dh_j  A|)
\leq
\int\limits_{E_0} \tau^{\lambda}(E^{\lambda}_{(s,+\infty)}(| \LL(\lambda) A(T_j(\lambda))|)) d\mu(\lambda).
\end{equation}
\end{prop}
Now, we formulate
\begin{prop}
\label{PROP:estimate_on_center}
Let $\M$ be a semifinite von Neumann algebra  with a fixed abelian subalgebra $Z$ such that its spectrum $\widehat{Z}$ allows for a Littlewood-Paley decomposition $E_j,T_j, \Psi_j$. Let $\Dh \in L^0(\M), A \in L^0(\Mh)$. Let $A, \Dh$ and $\widehat{\Psi_j}$ be as in Proposition \ref{PROP:compute_on_center}. Let us assume that the measurable fields $\lambda \mapsto \Mh(\lambda)$ and $\lambda \mapsto \LL(\lambda)$ are given. Then we have
\begin{equation}
\|\Dh_j  A\|_{L^{r,\infty}(\Mh)}
\leq \mu(E_0)^{r-\frac1{r}}
\| \|\LL_0 A(T_j(\lambda)\|_{L^{r,\infty}(\Mh_0)}\|_{L^{r,1}(\widehat{Z})}
.\quad \forall j \in \mathbb{Z}
\end{equation}
\end{prop}

Now, we can formulate version of Theorem \ref{THM:Mihlin-Hormander-semifinite} using elements of noncommutative Littlewood-Paley theory (see Proposition \ref{PROP:compute_on_center}).
\begin{thm}
\label{THM:Mihlin-Hormander-general-noncomm-L-Paley}
Let $(\M,\Mh)$ be a tuple of  semi-finite von Neumann algebras admitting a Fourier structure.  Let $\widehat{\D} \colon L^2(\M) \to L^2(\M)$ be a linear operator affiliated with $\M$ and let $A\colon L^2(\M) \to L^2(\M)$ be a linear operator affiliated with $\Mh$.
Then for every $x \in L^{p}(\M)$ we have
\begin{equation}
\label{EQ:MH-Lp-Lp-Littlewood-Paley}
\|A x\|_{L^{p}(\M)}
\lesssim
\sup_{j\in\mathbb{Z}}
\| \|\LL(\lambda) A(T_j(\lambda)\|_{L^{r,\infty}(\Mh_{\lambda})}\|_{L^{r,1}(\widehat{Z})}
\|x\|_{L^{p}(\M)},
\end{equation}
where 
\begin{equation}
\frac{1}{r}=2\left|\frac{1}{p}-\frac{1}{2}\right|
\end{equation}
and 
\begin{equation}
\Mh = \bigoplus\limits_{\widehat{Z}}\int \Mh_{\lambda}d\mu(\lambda)
\end{equation}
\end{thm}

Now, we test the general theory in the classical case of the Euclidean harmonic analysis on $\mathbb{R}^n$, i.e.
\begin{exa}
\label{EX_case_of_Rn}
We show that Theorem \ref{THM:Mihlin-Hormander-general-noncomm-L-Paley} in the case of $\M=\mathbb{R}^n, \Mh=\widehat{\mathbb{R}}^n$ recovers \cite[Theorem 1.1]{GS19}.

Take $\M=L^{\infty}(\mathbb{R}^n)$ and $\Mh=L^{\infty}(\widehat{\mathbb{R}^n})$, where $\widehat{\mathbb{R}^n}$ denotes the unitary dual of $\mathbb{R}^n$. Let $\Dh=M_{(1+|x|^2)^{s/2}}$, $s\geq 0$, be the operator of multiplication by $(1+|x|^2)^{s/2}$. Fix a measurable strictly positive function $\Psi$ with $\Psi_j(\xi)=\Psi(2^{-j}\xi)$ and $\sum_{j\in\mathbb{Z}}\Psi(2^{-j}\xi) = 1$ for $\xi \in X$. Let $\widehat{\Psi}$ be the operator of convolution with $\Psi$, and set $\LL(\lambda):=(I-\Delta_{\mathbb{R}^n})^{s/2}\widehat{\Psi_0}^{1/s}$. Fix the Littlewood--Paley decomposition on $\widehat{\mathbb{R}^n}$ by $E_j=(2^{j-1}, 2^{j+1})$ and $T_j(x)=2^{j}x$, $x\in\widehat{\mathbb{R}}$. By Theorem~\ref{THM:Mihlin-Hormander-general-noncomm-L-Paley}, for $A$ affiliated with $\VN_L(\mathbb{R}^n)$ we have
\[
\|A\|_{L^p(\mathbb{R}^n)\to L^p(\mathbb{R}^n)}
\;\leq\;
\sup_{j\in\mathbb{Z}}
\left\|
\left[(I-\Delta_{\mathbb{R}^n})^{\frac{s}{2}}\widehat{\Psi}_0\sigma(2^{j}\cdot)\right]
\right\|_{L^{r,\,\infty}(\mathbb{R}^{n})},
\]
where $\frac{1}{r}=2|\frac{1}{p}-\frac{1}{2}|$ (see \eqref{EX_Rn_r_p_q} below). We have $\|\cdot\|_{L^{r,\,\infty}(\mathbb{R}^{n})} \leq \|\cdot\|_{L^{\frac{n}{s},\,\infty}(\mathbb{R}^{n})}$ when $r \leq \frac{n}{s}$, and by the Lorentz embedding $L^{\frac{n}{s},\infty}(\mathbb{R}^{n}) \hookrightarrow L^{\frac{n}{s}, 1}(\mathbb{R}^{n})$,
\begin{equation}
\label{EQ_grafakos_weaker_step_1}
\left\|(I-\Delta_{\mathbb{R}^{n}})^{\frac{s}{2}}\left[\widehat{\Psi} \sigma(2^{j} \cdot)\right]\right\|_{L^{\frac{n}{s}, \infty}(\mathbb{R}^{n})}
\leq
\left\|(I-\Delta_{\mathbb{R}^{n}})^{\frac{s}{2}}\left[\widehat{\Psi} \sigma(2^{j} \cdot)\right]\right\|_{L^{\frac{n}{s}, 1}(\mathbb{R}^{n})}.
\end{equation}
Thus the condition $\frac{s}{n} \geq \frac{1}{r} = 2|\frac{1}{p}-\frac{1}{2}|$ is sufficient for $L^p$-boundedness, i.e.\ we recover the sharp Grafakos--Slav\'ikov\'a condition \cite[Theorem 1.1]{GS19}. The relations are
\begin{equation}
\label{EX_Rn_r_p_q}
\frac{1}{r}=2\left|\frac{1}{p}-\frac{1}{2}\right|,\qquad r \leq \frac{n}{s}.
\end{equation}
\end{exa}

\section{Applications}
\label{sec:applications}
Let $\M$ admit a Fourier structure, that is, there exists a triple $\left(\Mh, \mathcal{F}_{1}, \mathcal{F}_{2}\right)$ satisfying Definition \ref{Fourier-structure}.

\begin{thm}\cite[Theorem 6.1]{AR}\label{Thm 6.1} Let $\LL$ be a closed unbounded operator affiliated with a semifinite von Neumann algebra $\M \subset B(\HH)$. Assume that $\varphi$ is a monotonically decreasing continuous function on $[0,+\infty)$ such that

\begin{align*}
\varphi(0) & =1,  \\
\lim _{u \rightarrow+\infty} \varphi(u) & =0.
\end{align*}
Then for every $1 \leq r<\infty$ we have the equality
\begin{equation*}\label{6.3}
\|\varphi(|\LL|)\|_{L^{r, \infty}(\M)}=\sup _{u>0}\left(\tau\left(E_{(0, u)}(|\LL|)\right)\right)^{\frac{1}{r}} \varphi(u).
\end{equation*}
\end{thm}
Let $\LL$ be an arbitrary unbounded linear operator affiliated with $(\M, \tau)$. Then Theorem \ref{Thm 6.1} says that the function $\varphi(|\LL|)$ is necessarily affiliated with $(\M, \tau)$ and $\varphi(|\LL|) \in L^{r, \infty}(\M)$ if and only if the $r$-th power $\varphi^{r}$ of $\varphi$ grows at infinity not faster than $\frac{1}{\tau\left(E_{(0, u)}(|\LL|)\right)}$, i.e. if we have the estimate
\begin{equation*}\label{6.4}
\varphi(u)^{r} \lesssim \frac{1}{\tau\left(E_{(0, u)}(|\LL|)\right)}, \quad u>0.
\end{equation*}
\begin{prop}
\label{prop_lorentz_norm}
Let $\M$ admit a Fourier structure $\left(\Mh, \mathcal{F}_{1}, \mathcal{F}_{2}\right)$. Let $\LL$ be affiliated with $(\Mh, \widehat{\tau})$. Assume that $\varphi$ is as in Theorem \ref{Thm 6.1}. Then we have the inequality

\begin{equation}\label{6.5}
\|\varphi(\LL)\|_{L^{r,\infty}(\Mh)} \leq \sup _{s>0} \varphi(s)\left[\tau\left(E_{(0, s)}(|\LL|)\right)\right]^{\frac1r},
\end{equation}
where $r\geq 1$
\end{prop}
\begin{proof}
The proof follows immediately from combining \cite[Theorem 3.3]{RT23} and Theorem \ref{Thm 6.1}
\end{proof}

Now, we demonstrate the connection of the spectral properties of the operators with the time decay rates of the propagators for the corresponding evolution equations. This is illustrated in the context of the heat equation, where the functional calculus and the application of Theorem~\ref{Thm 6.1} to the family of functions $\{e^{-ts}\}$ for $t>0$ provide the time decay rate for the solution $u= u(t)$ to the heat equation
\begin{equation}\label{evol-equ}\partial_t u(t) + \LL u(t) = 0, \quad u(0) = u_0.
\end{equation}
Let us assume that 
\begin{equation}
\label{EQ_weyl_law}
\tau(E_{(0,s)}(\LL) \cong s^{\alpha}
\end{equation}
For each $t>0,$ we apply Borel functional calculus \cite[Section 2.6]{AR} to obtain
\begin{equation}\label{heat-semigroup}
u(t)=e^{-t\LL}u_{0}.
\end{equation}
It can be verified that $u(t)$ satisfies equation \eqref{evol-equ} and the initial condition. Then, using Theorem \ref{THM:Mihlin-Hormander-semifinite-p-2-q} and Proposition \ref{prop_lorentz_norm}, we derive
\begin{equation}
\|e^{-t \LL} u_0\|_{L^q(\M)}
\leq
\|e^{-t \LL}\|_{L^{r,\infty}(\Mh)}
\|u_0\|_{L^p(\M)}.
\end{equation}
By elementary calculus, we can show
\begin{equation}
\|e^{-t \LL}\|_{L^{r,\infty}(\Mh)} = \sup_{s>0} e^{-ts}s^{\alpha (\frac1p-\frac1q)} = C\, t^{-\alpha (\frac1p-\frac1q)}
\end{equation}
We call equation $\frac{\partial^2{u}}{{\partial t^2}} + \LL u = 0, \partial u=u_1$ with boundary condition $u(0) = u_0$ an $\LL$-wave equation.
It is trivial that its solution is given by:
\[
u(t) = \frac{\sin(t\cdot \sqrt{\LL})}{\sqrt{\LL}} u_1 + \cos(t\sqrt{\LL}) u_0.
\]

\begin{thm}[Time decay for wave propagators]
\label{THM_time_rescaling}
Let $\Mh$ be a semi-finite von Neumann algebra and $\LL \in L^0(\Mh)$. Let us assume that $sp(\bigl|\LL\bigr|) \subset (\kappa, +\infty), \kappa> 0$. We assume
\begin{equation}
\label{EQ_weyl_law_alpha}
\tau(E_{(\kappa,s)}(\LL)) \cong s^{\alpha},\quad \alpha >0.
\end{equation}
Let us assume
\begin{equation}
\label{EQ_time_rescaling_r_alpha_beta}
\frac1r \geq \frac{\beta}{\alpha}.
\end{equation}
Let us fix $0<\gamma \leq 2$.
Then we have
\begin{equation}
\Bigg\|
\LL^{-\beta}\cdot \frac{\sin(t\sqrt{\LL})}{\sqrt{\LL}}
\Bigg\|_{L^{r,\,\infty}(\Mh)}
\leq  C_{\gamma,p,q,\beta,\kappa, \LL}\cdot t^{1-\gamma\beta -\frac{\alpha}{\beta}\gamma \frac1r}.
\end{equation}
\end{thm}
\begin{proof}[Proof of Theorem \ref{THM_time_rescaling}]
\begin{equation}
\label{eq_time_rescaling_proof_step_1}
\Bigg\|
\LL^{-\beta}\cdot \frac{\sin(t\sqrt{\LL})}{\sqrt{\LL}}
\Bigg\|_{L^{r,\,\infty}(\Mh)}
=\sup_{s>0} s \Bigg[\tau(E_{(s,+\infty)}\Bigg(\Bigg|\,\LL^{-\beta}\cdot \frac{\sin(t\sqrt{\LL})}{\sqrt{\LL}}\Bigg|\Bigg)\Bigg]
\end{equation}

Let us introduce dilated operator $\mathcal{Z}_t$ as follows
\begin{equation}
\label{eq_time_rescaling_proof_step_2}
\ZZ_t = t^{-\gamma} \LL, \gamma >0, t>0.
\end{equation}
Let us denote
\begin{equation}
\label{eq_time_rescaling_proof_step_3}
\sinc(x) =
\begin{cases}
\dfrac{\sin x}{x}, & x \ne 0,\\
1, & x = 0.
\end{cases}
\end{equation}
By the spectral mapping theorem, we immediately obtain
\begin{equation}
\label{eq_time_rescaling_proof_step_4}
\begin{aligned}
E_{(s,\infty)}\Bigl(\Bigl|\LL^{-\beta}\tfrac{\sin \bigl(t\sqrt{\LL}\bigr)}{\sqrt{\LL}}\Bigr|\Bigr)
&=
E_{(s,\infty)}\Bigl(t^{1-\gamma\beta }\,\Bigl|
\ZZ^{-\beta}_t
\sinc(t^{1-\frac{\gamma}{2}}\sqrt{\ZZ_t})
\bigr|\Bigr)\\
&= E_{(s t^{\gamma\beta-1},\infty)}
\Bigl(\bigl|
\ZZ^{-\beta}_t
\sinc(t^{1-\frac{\gamma}{2}}\sqrt{\ZZ_t})
\bigr|\Bigr).
\end{aligned}
\end{equation}
Based on \eqref{eq_time_rescaling_proof_step_4}, we can calculate
\begin{equation}
\label{eq_time_rescaling_proof_step_4b}
\begin{aligned}
&\sup_{s>0} s \Bigg[\tau(E_{(s,+\infty)}\Bigg(\Bigg|\,\LL^{-\beta}\cdot \frac{\sin(t\sqrt{\LL})}{\sqrt{\LL}}\Bigg|\Bigg)\Bigg]
&=\\
&\sup_{s>0} s \tau\Bigl(E_{(s t^{\gamma\beta-1},\infty)}\Bigl(\bigl|\ZZ^{-\beta}_t
\sinc(t^{1-\frac{\gamma}{2}}\sqrt{\ZZ_t}
\bigr|\Bigr)\Bigr)
&=\\
&t^{1-\gamma\beta} \sup_{s>0} s \tau\Bigl(E_{(s,+\infty)}\Bigl(\bigl|
\ZZ^{-\beta}_t
\sinc(t^{1-\frac{\gamma}{2}}\sqrt{\ZZ_t}
\bigr|\Bigr)\Bigr)
\end{aligned}
\end{equation}

Thus, we showed that
\begin{equation}
\label{eq_time_rescaling_proof_step_5}
\Bigg\|
\LL^{-\beta}\cdot \frac{\sin(t\sqrt{\LL})}{\sqrt{\LL}}
\Bigg\|_{L^{r,\,\infty}(\Mh)}
=t^{1-\gamma\beta}
\Bigg\|
\ZZ_t^{-\beta}
\sinc(t^{1-\frac{\gamma}{2}}\sqrt{\ZZ_t})
\Bigg\|_{L^{r,\,\infty}(\Mh)}.
\end{equation}

By H\"older inequality, we get
\begin{equation}
\label{eq_time_rescaling_proof_step_6}
\Bigg\|
\ZZ_t^{-\beta}
\sinc(t^{1-\frac{\gamma}{2}}\sqrt{\ZZ_t})
\Bigg\|_{L^{r,\,\infty}(\Mh)}
\leq
\Bigg\|
\ZZ^{-\beta}_t
\Bigg\|_{L^{r,\infty}(\Mh)}
\Bigg\|
\sinc(t^{1-\frac{\gamma}{2}}\sqrt{\ZZ_t})
\Bigg\|_{L^{\infty,\infty}(\Mh)}
\end{equation}
We notice that $\sinc(t^{1-\frac{\gamma}{2}}\sqrt{\ZZ_t})$ is a bounded operator for any $t>0, \gamma\leq 2$, and we can estimate
\begin{equation}
\label{eq_time_rescaling_proof_step_7}
\Bigg\|
\sinc(t^{1-\frac{\gamma}{2}}\sqrt{\ZZ_t})
\Bigg\|_{L^{\infty,\infty}(\Mh)}
\leq C_{\gamma,\LL} < +\infty.
\end{equation}
By the spectral mapping theorem, we have
\begin{equation}
\label{eq_time_rescaling_proof_step_8}
\begin{aligned}
&E_{(s,+\infty)}(\ZZ^{-\beta}_t) \\
=
&E_{(s,+\infty)}(t^{\gamma\beta}\LL^{-\beta}) \\
=
&E_{(0,s^{-\frac1{\beta}})} (t^{\gamma\beta} \LL) \\
&E_{(0,s^{-\frac1{\beta}} t^{-\gamma\beta})} ( \LL) \\
=&(s^{-\frac1{\beta}} t^{-\gamma\beta})^{\alpha}
\cong s^{-\frac{\alpha}{\beta}}t^{-\gamma\beta\alpha},\quad s\geq \kappa.
\end{aligned}.
\end{equation}
Thus, we get  from \eqref{eq_time_rescaling_proof_step_8}
\begin{equation}
\label{eq_time_rescaling_proof_step_9}
\|\ZZ^{-\beta}_t\|_{L^{r,\infty}(\Mh)}
=
\sup_{s>0} s[s^{-\frac{\alpha}{\beta}}t^{-\gamma\beta\alpha}]^{\frac1r}
=t^{-\alpha\beta\gamma \frac1r}
\sup_{s>\kappa} s^{1-\frac{\alpha}{r\beta}} = t^{-\frac{\alpha}{\beta}\gamma \frac1r} \kappa^{1-\frac{\alpha}{\beta r}}
\end{equation}
Collecting \eqref{eq_time_rescaling_proof_step_5} and \eqref{eq_time_rescaling_proof_step_7} and \eqref{eq_time_rescaling_proof_step_9}, we finally obtain
\begin{equation}
\begin{aligned}
\Bigg\|
\LL^{-\beta}\cdot \frac{\sin(t\sqrt{\LL})}{\sqrt{\LL}}
\Bigg\|_{L^{r,\,\infty}(\Mh)} \\
& \leq C_{\gamma,\LL}\cdot
t^{1-\gamma\beta} t^{-\frac{\alpha}{\beta}\gamma \frac1r} \kappa^{1-\frac{\beta}{r}}\\
&=
 C_{\gamma,p,q,\beta,\kappa, \LL}\cdot t^{1-\gamma\beta -\frac{\alpha}{\beta}\gamma \frac1r}.
 \end{aligned}
\end{equation}

\end{proof}

\section{Acknowledgements}
The second and third authors were partially supported by Odysseus and Methusalem grants (01M01021 (BOF Methusalem) and 3G0H9418 (FWO Odysseus)) from Ghent Analysis and PDE center at Ghent University. The second author was also supported by the EPSRC grants EP/R003025/2 and EP/V005529/1.
\printbibliography

\end{document}